\documentclass[11pt,dvips,twoside,letterpaper]{article}
\usepackage{pslatex}
\usepackage{fancyhdr}
\usepackage{graphicx}
\usepackage{geometry}
\RequirePackage[latin1]{inputenc}
 \RequirePackage[T1]{fontenc}

\def\figurename{Figure} 
\makeatletter
\renewcommand{\fnum@figure}[1]{\figurename~\thefigure.}
\makeatother

\def\tablename{Table} 
\makeatletter
\renewcommand{\fnum@table}[1]{\tablename~\thetable.}
\makeatother

\usepackage{amsmath}
\usepackage{amssymb}
\usepackage{amsfonts}
\usepackage{amsthm,amscd}

\newtheorem{theorem}{Theorem}[section]
\newtheorem{lemma}[theorem]{Lemma}
\newtheorem{corollary}[theorem]{Corollary}
\newtheorem{proposition}[theorem]{Proposition}
\theoremstyle{definition}
\newtheorem{definition}[theorem]{Definition}

\theoremstyle{remark}
\newtheorem{remark}[theorem]{Remark}

\numberwithin{equation}{section}

\def\P{\mathbb P}

\def\R{\mathbb R}
\def\E{\mathbb E}

\def\E{\mathbb E}

\begin{document}
\title{ \textbf{Multivalued stochastic Dirichlet-Neumann problems and
generalized backward doubly stochastic differential
equations}\footnote{The two first authors are supported by the
National Natural Science Foundation of
 China (No. 10901003), the Great Research Project of Natural
 Science Foundation of Anhui Provincial Universities (No. KJ2010ZD02)
 and the Chinese Universities Scientific Fund
(No. BUPT2009RC0705)\newline The third author is supported by TWAS Research Grants to individuals (No. 09-100 RG/MATHS/AF/\mbox{AC-I}--UNESCO FR: 3240230311)}
}

\author{\textbf{Yong
Ren}$^1$\footnote{Corresponding author. e-mail: renyong@126.com}\;\,,\ \
\textbf{ Qing Zhou}$^2$ \footnote{ e-mail: zqleii@gmail.com}\;\,, \ \
\textbf{Auguste Aman}$^3$\footnote{e-mail: augusteaman5@yahoo.fr;\, auguste.aman@univ-cocody.ci}
\\
\small 1. Department of Mathematics, Anhui Normal University, Wuhu
241000, China\;\;\;\;\;\;\;\;\;\;\;\;\;\;\;\;\;\;\;\;\;\;\;\;\;\;\\
 \small  2. School of Science, Beijing University of Posts and Telecommunications, Beijing 100876, China   \\
\small 3. U.F.R Mathematiques et Informatique, Universit\'{e} de
Cocody, 22 BP 582 Abidjan, C\^{o}te d'Ivoire}\,
\date{}
\maketitle

\begin{abstract}
In this paper, a class of generalized backward doubly stochastic
differential equations whose coefficient contains the
subdifferential operators of two convex functions, which are also called
as generalized backward doubly stochastic variational inequalities, are
considered. By means of a penalization argument based on Yosida
approximation, we establish the existence and uniqueness of the
solution. As an application, this result is used to derive
existence result of stochastic viscosity solution for a class of
multivalued stochastic Dirichlet-Neumann problems.
\end{abstract}

\noindent {\bf 2000 MR Subject Classification} 60H15, 60H20

\vspace{.08in} \noindent \textbf{Keywords} Backward doubly
stochastic differential equation, subdifferential operator, Yosida
approximation, Neumann-Dirichlet boundary conditions, stochastic
viscosity solution

\section{Introduction}

The theory of nonlinear backward stochastic differential equations
(BSDEs, for short) was firstly developed by Pardoux and Peng
\cite{PP1}. These equations have attracted great interest due to
their connections with stochastic control, mathematical finance and
due to providing probabilistic interpretation for solutions of PDEs. One
can see  Hamed\`{e}ne and Lepeltier \cite{HL}, El Karoui et al.
\cite{El}, Peng \cite{Peng}, Ren et al. \cite{ren1,ren2} and the
references therein. Further, other settings of BSDEs have been
introduced. Especially, Gegout-Petit and Pardoux \cite{Geg}
introduced a class of BSDEs related to a multivalued maximal
monotone operator defined by the subdifferential operator of a
convex function. In addition, Pardoux and R\u{a}\c{s}canu \cite{PR}
proved the existence and uniqueness of the solution of BSDEs, on a
random (possibly infinite) time interval, involving a
subdifferential operator in order to give a probabilistic
interpretation for the viscosity solution of some parabolic and
elliptic variational inequalities. Its extension to the probabilistic
interpretation of the viscosity solution of the parabolic
variational inequality (PVI, for short) with a mixed nonlinear
multivalued Neumann-Dirichlet boundary condition was recently given
in Maticiuc and R\u{a}\c{s}canu \cite{MR}.

Another class of BSDEs, named backward doubly stochastic differential equations (BDSDEs, for short) involving a standard
forward stochastic integral and a backward stochastic integral has been proposed by Pardoux and Peng \cite{PP}. They derive
existence and uniqueness result under global Lipschitz assumptions on the coefficients and use it under stronger assumptions
(coefficients are $C^3$) to give a probabilistic representation for a class of quasilinear stochastic partial differential equations (SPDEs, for short). Furthermore, Buckdahn and Ma \cite{BM1,BM2,BMa} improve this representation introducing the viscosity solution of semi-linear SPDEs. This viscosity solution is been extended to semi-linear SPDE with a Neumann boundary in \cite{Bal} by means of a class of generalized backward doubly stochastic differential equations (GBDSDEs, for short). On the other hand, Boufoussi and Mrhardy \cite{BM} derive the existence result to stochastic viscosity solution for some multivalued SPDE using it connection with
BDSDEs whose coefficient contains the subdifferential of a convex function. More recently, the scalar GBDSDE with one-sided reflection which provides a probabilistic representation for the stochastic viscosity solution of an obstacle problem for a nonlinear stochastic parabolic PDE was considered by Aman and Mrhardy in \cite{AM}.

Motivated by the above works, the purpose of the present paper is to consider the following generalized backward doubly stochastic
differential equation, whose coefficient contains the subdifferential operators of two convex functions, also called generalized  backward
doubly stochastic variational inequality (BDSGVI, for short). Precisely, we have
\begin{eqnarray}\label{equation1}
\left\{
\begin{array}{ll}
\,{\rm d}Y_t+f(t,Y_t,Z_t)\,{\rm d}t+g(t,Y_t)\,{\rm
d}A_t+h(t,Y_t,Z_t)\,{\rm d}B_t
\\\\~~~~~~~~~~~~~~~~~~~~~~~~\in \partial
\varphi(Y_t)\,{\rm d}t+\partial \psi(Y_t)\,{\rm d}A_t+Z_t\,{\rm
d}W_t
,\ 0\leq t\leq T, \\
Y_T=\xi,
\end{array}\right.
\end{eqnarray}
where $(A_t)_{t\geq 0}$ is a one-dimensional continuous increasing
$\mathcal{F}_t$-progressively measurable process, $\partial \varphi$
and $\partial \psi$ are two subdifferential operators. The
integral with respect to $\{B_t\}$ is a backward Kunita-It\^{o}
integral (see Kunita \cite{Kunita}) and this one with respect to
$\{W_t\}$ is a standard forward It\^{o} integral (see Gong
\cite{Gong}). It is actually a class of GBDSDEs, which involves two
subdifferential operators of two convex functions. Let us recall that \eqref{i1} has been studied, in the case $h=0$, in \cite{MR}.

We have two goals. First, under Lipschitz conditions on $f,\, g$ and $h$, we derive an existence and uniqueness
result to BDSGVI \eqref{equation1} by means of the Yosida approximation. Next, we naturally establish the connection between solution of \eqref{equation1} and the stochastic viscosity solution of the following stochastic PVI (SPVI, for short) with a mixed nonlinear multivalued Neumann-Dirichlet boundary condition:
\begin{eqnarray}
\left\{
\begin{array}{ll}
\Big(\frac{\partial u}{\partial t}(t,x)+Lu(t,x)+f(t,x,u(t,x),(\nabla u\sigma)(t,x))\\
~~~~~~~~~~~~~~~-h(t,x,u(t,x))\dot{B}_t\Big)\in\partial \varphi(u(t,x)),\;\; \ (t,x)\in [0,T]\times \Theta,\\
\frac{\partial u}{\partial n}(t,x)+g(t,x,u(t,x))\in\partial \psi(u(t,x)), \
(t,x)\in [0,T]\times Bd(\Theta),\\\\
u(T,x)=\chi(x),\;\;x\in\overline{\Theta}.
\end{array}\right.
\label{i1}
\end{eqnarray}
Here $\dot{B}$ denotes the white noise and, thus, indicates that the differential is to be understood in It\^{o}'s backward integral
sense with respect to Brownian motion $B$ and $f,\, h,\, g $ and $\chi$ are some measurable functions with appropriate dimensions. Moreover, $L$ and $\frac{\partial}{\partial n}$ denote the infinitesimal generator of some reflected diffusion process and are defined by
\begin{eqnarray*}
L=\frac{1}{2}\sum_{i,j=1}^{d}(\sigma(x)\sigma^{*}(x))_{i,j}
\frac{\partial^{2}}{\partial x_{i}\partial
x_{j}}+\sum^{d}_{i=1}b_{i}(x)\frac{\partial}{\partial x_{i}},\quad
\forall\, x\in\Theta,
\end{eqnarray*}
and
\begin{eqnarray*}
\frac{\partial}{\partial n}=\sum_{i=1}^{d}\frac{\partial\phi(x) }{\partial
x_{i}}(x)\frac{\partial}{\partial x_{i}},\quad \forall\,
x\in\partial\Theta,
\end{eqnarray*}
where the function $\phi$ is linked to a connected bounded domain $\Theta$ as defined in \cite{LZ}. Let us mention that \eqref{i1} is the generalization of the existence result (and not of the uniqueness) for the solution of the equation studied in \cite{MR}, where $h=0$ and $f, g$ are not random. As in \cite{BMa} or \cite{BM}, we shall define stochastic viscosity solution for SPVI \eqref{i1} by using the notion of stochastic sub- and super-jets. But the novelty lies in adding Stieltjes integrals with respect the process $A$ which allows us to give  representation formula (Feynman-Kac formula) for solution of the stochastic multivalued Neumann-Dirichlet problems. Therefore in our mind, the results of this article is a non trivial generalization of the work from \cite{BM} and hence one appear in \cite{BMa}.

The rest of this paper is organized as follows. In Section 2, we give some preliminaries and notations. Section 3 is devoted to prove the existence and uniqueness of a solution for the BDSGVI \eqref{equation1}. As application, we derive the notion of stochastic viscosity solution for a class of multivalued stochastic Dirichlet-Neumann problems and then prove its existence in the last section.

\section{ Preliminaries and notations}
In the sequel, let $T>0$ be a fixed terminal time, $\{B_t:t\in
[0,T]\}$ and $\{W_t:t\in [0,T]\}$ be two independent $d$-dimensional
Brownian motions ($d\geq 1$), defined on the complete probability
$(\Omega,\mathcal{F},\P)$ and $(\Omega',\mathcal{F}',\P')$
respectively, and $(A_t)_{t\geq 0}$ be a one-dimensional continuous
increasing measurable stochastic process (m.s.p.,
for short). Let us consider the product space $(\bar{\Omega},
\bar{\mathcal{F}}, \bar{\P}),$ where
\begin{eqnarray*}
\bar{\Omega}=\Omega\otimes\Omega',\
\bar{\mathcal{F}}=\mathcal{F}\otimes\mathcal{F}',\
\bar{\P}=\P\otimes\P',
\end{eqnarray*}
and let $\mathcal{N}$ denote the totality of $\bar{\mathbb{P}}$-null sets
of $\bar{\mathcal{F}}$. For each $t\in [0,T]$, we define
\begin{eqnarray*}
\mathcal{F}_t=\mathcal {F}_{t,T}^B\vee \mathcal {F}_{t}^W\vee
\mathcal{N},
\end{eqnarray*}
where for any process $\{\eta_t\}, \mathcal
{F}_{s,t}^\eta=\sigma\{\eta_r-\eta_s:s\leq r\leq t\},\mathcal {F}_{t}^\eta=\mathcal {F}_{0,t}^\eta.$ We note
that ${\bf F}= \{\mathcal{F}_{t},\ t\in [0,T]\}$ is neither
increasing nor decreasing so that it does not constitute a
filtration. Further, we assume that random variables
$\xi(\omega),\;\omega\in \Omega$ and $\zeta(\omega'),\; \omega'\in
\Omega'$ are considered as random variables on $\bar{\Omega} $ via
the following identifications:
\begin{eqnarray*}
\xi(\omega,\omega')=\xi(\omega),\
\zeta(\omega,\omega')=\zeta(\omega').
\end{eqnarray*}
In what follows, we will work under the following spaces as defined in \cite{MR}. For the
positive constants $\lambda$ and $\mu$,
\begin{itemize}
\item $\mathcal{M}_k^{\lambda,\mu}$ denotes the space of ${\bf F}$-jointly measurable stochastic process (j.m.s.p, in short) $\gamma:\Omega
\times [0,T]\to \mathbb{R}^k$ such that
  $$\|\gamma\|_{\mathcal{M}}^2=\mathbb{E}\int_0^T{\rm e}^{\lambda t+\mu A_t}|\gamma(t)|^2\,{\rm d}t<\infty.$$
\item $\bar {\mathcal{M}}_k^{\lambda,\mu}$ denotes the space of j.m.s.p. $\gamma:\Omega
\times [0,T]\to \mathbb{R}^k$ such that
 $$\|\gamma\|_{\bar{\mathcal{M}}}^2=\mathbb{E}\int_0^T{\rm e}^{\lambda t+\mu A_t}|\gamma(t)|^2\,{\rm d}A_t<\infty.$$

\item $\mathcal{S}_k^{\lambda,\mu}$ denotes the space of continuous j.m.s.p. $\gamma:\Omega
\times [0,T]\to \mathbb{R}^k$ such that
  $$\|\gamma\|_{\mathcal{S}}^2=\mathbb{E}\left(\sup_{0\leq t\leq T}{\rm e}^{\lambda t+\mu A_t}|\gamma(t)|^2\right)<\infty.$$
\end{itemize}
Now, we give the following assumptions:
\begin{itemize}
 \item[(H1)] The ${\bf F}^{B}$-adapted functions $f:[0,T]\times \Omega \times \mathbb{R}^k\times
 \mathbb{R}^{k\times d}
 \to \mathbb{R}^k,\ h:[0,T]\times \Omega \times \mathbb{R}^k\times
 \mathbb{R}^{k\times d}
 \to \mathbb{R}^{k\times d}$ and $g:[0,T]\times \Omega \times \mathbb{R}^k
 \to \mathbb{R}^{k}$
satisfy, for some constants $\beta_1,\ \beta_2 \in \mathbb{R},\
K>0,\ 0<\alpha<1,$ three ${\bf F}^{B}$-adapted processes
$\{f_t,g_t,h_t:0\leq t \leq T\}$, and for all $t\in [0,T], y,y' \in
\mathbb{R}^k $ and $z,z'\in \mathbb{R}^{k\times d}, $
\begin{itemize}
\item[(i)] $f(\cdot,\cdot,y,z), \ h(\cdot,\cdot,y,z)$ and $g(\cdot,\cdot,y)$ are p.m.s.p.;

\item[(ii)] $|f(t,y,z)|\leq f_t+K(|y|+\|z\|)$,\;\;  $\|h(t,y,z)\|\leq h_t+K(|y|+\|z\|),$\;\; $|g(t,y)|\leq g_t+K|y|;$

\item[(iii)] $f_t, h_t \in \mathcal{M}_k^{\lambda,\mu}$ and $g_t \in
\bar{\mathcal{M}}_k^{\lambda,\mu};$

\item[(iv)] $\left<y-y^\prime,f(t,y,z)-f(t,y^\prime,z)\right>\leq
\beta_1|y-y^\prime|^2$;

\item[(v)] $|f(t,y,z)-f(t,y,z^\prime)|\leq K\|z-z^\prime\|$;

\item[(vi)]  $\left<y-y^\prime, g(t,y)-g(t,y^\prime)\right>\leq
\beta_2|y-y^\prime|^2$;

\item[(vii)] $\|h(t,y,z)-h(t,y^\prime,z^\prime)\|^2\leq
K|y-y^\prime|^2+\alpha\|z-z^\prime\|^2$;

\item[(viii)] $y\mapsto (f(t,y,z),h(t,y,z),g(t,y))$ is continuous for all
$z, (t,\omega)$ a.e.

\end{itemize}

 \item[(H2)] The two functions $\varphi$ and $\psi$ satisfy
 \begin{itemize}
\item [(i)] $\varphi, \psi:\mathbb{R}^k\rightarrow (-\infty,+\infty]$ are proper
($\neq\infty$), convex, and lower semi continuous (l.s.c., for
short),
\item [(ii)] $\varphi(y)\geq \varphi(0)=0,\ \psi(y)\geq \psi(0)=0 $.

\end{itemize}
\item[(H3)] The terminal value $\xi$ is an $\mathbb{R}^k$-valued
 $\mathcal{F}_T$-measurable random variable. Moreover we have the following: For each  $\lambda>0,\mu>0$, satisfying $\lambda>2+2(\beta_1+\beta_2)+\frac{K(3-\alpha+2K)}{1-\alpha}$ and $\mu>1+2\beta_2$,
\begin{eqnarray*}
\Lambda:=\mathbb{E}\left\{{\rm e}^{\lambda T+\mu A_T}\left(|\xi|^2+\varphi(\xi)+\psi(\xi)\right)+\int_0^T
{\rm e}^{\lambda t+\mu A_t}\left[\left(|f_t|^2+|h_t|^2\right)\,{\rm
d}t+|g_t|^2\,{\rm d}A_t\right]\right\}<\infty.\label{esti}
\end{eqnarray*}
\end{itemize}

For $\theta$ equal to $\varphi$ or $\psi$, let us define
\begin{eqnarray*}
\begin{array}{l}
{\rm Dom}(\theta)=\left \{u\in \mathbb{R}^k :\theta(u)<+\infty
\right\},\\\\
\partial \theta(u)= \{u^\ast\in
\mathbb{R}^k:\left<u^\ast,v-u\right>+\theta(u)\leq \theta(v),\
\mbox{for all}\ v\in \mathbb{R}^k\},\\\\
{\rm Dom}(\partial \theta)= \{u\in \mathbb{R}^k:\partial \theta
\neq \emptyset\},\\\\
(u,u^*)\in\partial \theta \Leftrightarrow u\in {\rm Dom} (\partial \theta), u^\ast \in\partial\theta(u).
\end{array}
\end{eqnarray*}
It is well known that the subdifferential operator $\partial\theta$ is a maximal monotone
operator, which means that
\begin{eqnarray*}
\langle u-u^*,v-v^*\rangle\geq 0,\;\; \forall\; (u,u^*),\,(v,v^*)\in\partial\theta.
\end{eqnarray*}
\begin{remark}
Assumption (H2-$ii$) is not a restriction since we can replace $\varphi(u)$ (resp.$\;\psi(u)$) by $\varphi(u+u_0)-\varphi(u_0)-\langle u^{*}_0, u\rangle$ (resp.$\;\psi(u+u_0)-\psi(u_0)-\langle u^{*}_0, u\rangle$) where
$(u_0; u^{*}_{0})\in\partial\varphi$ (resp. $\;(u_0; u^{*}_{0})\in\partial\psi$).
\end{remark}
To end this section, let us introduce the following needed classical
Yosida approximation of the subdifferential operator $\partial
\theta$ equal to $\partial\varphi$ or $\partial\psi$. For
$\varepsilon>0$, we define (see \cite{Br} and the references
therein)
\begin{eqnarray*}
\theta_\varepsilon(x)=\min_{y}\left(\frac{1}{2}|x-y|^2+\varepsilon\theta(y)\right)=
\frac{1}{2}\left|x-J_\varepsilon(x)\right|^2 +\varepsilon\theta(J_\varepsilon(x)),
\end{eqnarray*}
where $J_\varepsilon(x)=(I+\varepsilon\partial \theta)^{-1}(x)$ is called the resolvent of the monotone
operator $\partial \theta$. Next, on can show that
\begin{eqnarray*}
\nabla\theta_{\varepsilon}(x)=\frac{x-J_{\varepsilon}(x)}{\varepsilon},
\end{eqnarray*}
and $x\mapsto\nabla\theta_{\varepsilon}(x)$ is a monotone Lipschitz function.\newline
Now, let us give the compatibility assumptions, which appear for the first time in \cite{MR}:
\begin{itemize}
\item [(H4)](compatibility assumptions)\newline  For all $ y\in \mathbb{R}^k, z\in\mathbb{R}^{k\times
  d},\varepsilon>0 $ and $0\leq t\leq T,$
\begin{itemize}
\item[(i)] $\left<\nabla\varphi_\varepsilon(y),\nabla\psi_\varepsilon(y)\right>\geq
  0;$
\item [(ii)] $\left<\nabla\varphi_\varepsilon(y),g(t,y)\right>\leq
\left<\nabla\psi_\varepsilon(y),g(t,y)\right>^+;$
\item[(iii)] $\left<\nabla\psi_\varepsilon(y),f(t,y,z)\right> \leq
\left<\nabla\varphi_\varepsilon(y),f(t,y,z)\right>^+,$
 where $a^+=\max\{0,a\}.$
\end{itemize}
\end{itemize}
Recall again that $\theta$ is equal to $\varphi$ or $\psi$, we have (see \cite{Br} or \cite{PR} ).
\begin{proposition}\label{pro2.1}
\begin{enumerate}
\item[\rm (1)] $\theta_\varepsilon$ is a
convex function with Lipschitz continuous derivatives;
\item[\rm (2)] for all $ x\in \mathbb{R}^k,
\nabla\theta_\varepsilon(x)=\partial
\theta_\varepsilon(x)=\frac{1}{\varepsilon}\left(x-J_\varepsilon(x)\right)\in
\partial \theta(J_\varepsilon(x));$\\
\item[\rm (3)] for all $ x,y\in \mathbb{R}^k,
|\nabla\theta_\varepsilon(x)-\nabla\theta_\varepsilon(y)|\leq
\frac{1}{\varepsilon}
|x-y|;$\\
\item[\rm (4)] for all $ x,y\in \mathbb{R}^k,
\left<\nabla\theta_\varepsilon(x)-\nabla\theta_\varepsilon(y),x-y\right>\geq
0;$\\
\item[\rm (5)] for all $ x,y\in \mathbb{R}^k\ \mbox{and} \
\varepsilon,\delta>0,$
$\left<\nabla\theta_\varepsilon(x)-\nabla\theta_\delta(y),
x-y\right>\geq-\left(\varepsilon+\delta\right)\left<\nabla\theta_\varepsilon(x),\nabla\theta_\delta(y)\right>.
$
\end{enumerate}
\end{proposition}

\section{Existence and uniqueness result to BDSGVI}
This section aims to derive the existence and uniqueness result to
BDSGVI \eqref{equation1}. They are obtained via Yosida
approximations. First of all, let us introduce the adapted definition of solution from \cite{MR} to our BDSGVI.
\begin{definition}\rm \label{def2.1}
The processes $(Y,U,V,Z)$ is called a solution of BDSGVI \eqref{equation1} if the
following conditions are satisfied:
\begin{itemize}
  \item[(1)] $\displaystyle Y\in \mathcal{S}_k^{\lambda,\mu}\cap \mathcal{M}_k^{\lambda,\mu}
  \cap  \bar{\mathcal{M}}_k^{\lambda,\mu},\ Z\in
 \mathcal{M}_{k\times d}^{\lambda,\mu};$
  \item[(2)] $\displaystyle U\in \mathcal{M}_k^{\lambda,\mu}
  ,\ V\in
\bar{ \mathcal{M}}_{k}^{\lambda,\mu};$

\item[(3)] $\displaystyle \mathbb{E}\int_0^T{\rm e}^
{\lambda t+\mu A_t}\left[\varphi(Y_t)\,{\rm d}t+\psi(Y_t)\,{\rm
d}A_t\right]<\infty;$

   \item[(4)]$\displaystyle (Y_t,U_t)\in \partial \varphi, \ \,{\rm d}\overline{\mathbb{P}}\otimes
  \,{\rm d}t, \ (Y_t,V_t)\in \partial \psi, \ \,{\rm d}\overline{\mathbb{P}}\otimes
  dA_t(\bar{\omega})
  \mbox{-a.e. on}\ [0,T];$\\
  \item[(5)] $\displaystyle Y_t+\int_t^TU_s\,{\rm
d}s+\int_t^TV_s\,{\rm d}A_s =\xi+\int_t^Tf(s,Y_s,Z_s)\,{\rm
d}s+\int_t^Tg(s,Y_s)\,{\rm
d}A_s$$$~~~~~~~~~~~~~~~~~~~~~~~~~~~~~+\int_t^Th(s,Y_s,Z_s)\,{\rm
d}B_s-\int_t^TZ_s\,{\rm d}W_s,\ 0\leq t \leq T. $$
\end{itemize}
\end{definition}
Since our method is based on the Yosida approximations, let us
consider the following GBDSDE:
\begin{eqnarray}\label{eq1}Y_t^\varepsilon
+\int_t^T\nabla \varphi_\varepsilon(Y_s^\varepsilon)\,{\rm d}s
+\int_t^T\nabla \psi_\varepsilon(Y_s^\varepsilon)\,{\rm d}A_s
&=&\xi+\int_t^Tf(s,Y_s^\varepsilon,Z_s^\varepsilon)\,{\rm d}s
+\int_t^Tg(s,Y_s^\varepsilon)\,{\rm d}A_s\nonumber\\
&&+\int_t^Th(s,Y_s^\varepsilon,Z_s^\varepsilon)\,{\rm
d}B_s-\int_t^TZ_s^\varepsilon\,{\rm d}W_s.
\end{eqnarray}
 Since
$\nabla\varphi_\varepsilon $ and $\nabla\psi_\varepsilon$ are
Lipschitz continuous, it is known from a recent result of Boufoussi
et al. \cite{Bal}, that GBDSDE\eqref{eq1} has a unique solution $
(Y^\varepsilon,Z^\varepsilon)\in
 \left(\mathcal{S}_k^{\lambda,\mu} \cap\mathcal{M}_k^{\lambda,\mu}
 \cap  \bar{\mathcal{M}}_k^{\lambda,\mu}\right)\times
 \mathcal{M}_{k\times d}^{\lambda,\mu}$.

Setting $$(U_t^\varepsilon,V_t^\varepsilon)=(\nabla
\varphi_\varepsilon(Y_t^\varepsilon),\nabla
\psi_\varepsilon(Y_t^\varepsilon)), \ 0\leq t \leq T,$$ we shall prove the convergence of the sequence
$(Y^\varepsilon,U^\varepsilon,V^\varepsilon,Z^\varepsilon)$ to a
process $(Y,U,V,Z)$, which is the desired solution of the BDSGVI
\eqref{equation1}.

The principal result of this section is the following theorem. We would like to point out that the proofs of Lemma \ref{lemma1}, Lemma
\ref{lemma2}, Lemma \ref{lemma3} and Theorem \ref{thm3.1} are generalizations of the results from \cite{MR}.
For the reading convenience, we give the detailed calculations.
\begin{theorem}\label{thm3.1}
Assume the assumptions of \ {\rm(H1)--(H4)} hold. Then, the BDSGVI
\eqref{equation1} has a unique solution.
\end{theorem}
In the sequel, $C>0$ is a constant which can change its value from
line to line. Firstly, we give a prior estimates on the solution.
\begin{lemma}\label{lemma1}
Assume the assumptions of \ {\rm(H1)--(H3)} hold. Then, it holds
that
\begin{equation}\label{equation2}
\mathbb{E}\left[\sup_{0\leq t\leq T}{\rm e}^{\lambda t+\mu
A_t}|Y_t^\varepsilon|^2+\int_0^T {\rm e}^{\lambda t+\mu
A_t}\left(|Y_t^\varepsilon|^2(\,{\rm d}t+\,{\rm
d}A_t)+\|Z_t^\varepsilon\|^2\,{\rm d}t\right)\right]\leq C.
\end{equation}
\end{lemma}
\begin{proof}
Applying It\^{o}'s formula to ${\rm e}^{\lambda t+\mu
A_t}|Y_t^\varepsilon|^2$, we obtain

 $\displaystyle{\rm e}^{\lambda t+\mu A_t}|Y_t^\varepsilon|^2
 +\int_t^T{\rm e}^{\lambda s+\mu A_s}|Y_s^\varepsilon|^2\,{\rm d}(\lambda
s+\mu A_s)$ \begin{eqnarray*}&&+2\int_t^T{\rm e}^{\lambda s+\mu
A_s}\left[\left<Y_s^\varepsilon,\nabla
\varphi_\varepsilon(Y_s^\varepsilon)\right>\,{\rm
d}s+\left<Y_s^\varepsilon,\nabla
\psi_\varepsilon(Y_s^\varepsilon)\right>\,{\rm
d}A_s\right]+\int_t^T{\rm e}^{\lambda s+\mu
A_s}\|Z_s^\varepsilon\|^2\,{\rm d}s\\&=&{\rm e}^{\lambda T+\mu
A_T}|\xi|^2+2\int_t^T{\rm e}^{\lambda s+\mu
A_s}\left<Y_s^\varepsilon,f(s,Y_s^\varepsilon,Z_s^\varepsilon)\right>\,{\rm
d}s+2\int_t^T{\rm e}^{\lambda s+\mu
A_s}\left<Y_s^\varepsilon,g(s,Y_s^\varepsilon)\right>\,{\rm
d}A_s\\&& +2\int_t^T{\rm e}^{\lambda s+\mu
A_s}\left<Y_s^\varepsilon,h(s,Y_s^\varepsilon,Z_s^\varepsilon)\,{\rm
d}B_s\right> +\int_t^T{\rm e}^{\lambda s+\mu
A_s}\|h(s,Y_s^\varepsilon,Z_s^\varepsilon)\|^2\,{\rm
d}s\\&&-2\int_t^T{\rm e}^{\lambda s+\mu
A_s}\left<Y_s^\varepsilon,Z_s^\varepsilon \,{\rm
d}W_s\right>.
\end{eqnarray*}
Using the elementary inequality
$2ab\leq\beta^2 a^2+\frac{b^2}{\beta^2}$, for all $a,b\geq 0,$
 and (H1), we get
\begin{itemize}
\item []$\begin{array}{rcl}\displaystyle 2\left<y,f(s,y,z)\right>&=&2\left<y,f(s,y,z)-f(s,y,0)+f(s,y,0)-f(s,0,0)+f(s,0,0)\right>\\\\
&\leq& 2\beta_1|y|^2+2K|y|\|z\|+|y|^2+|f(s,0,0)|^2\\\\
&\leq&\left(1+2\beta_1+\frac{K^2}{M}\right)|y|^2+M\|z\|^2+|f(s,0,0)|^2,
 \end{array}$

\item [] $\begin{array}{rcl}\|h(s,y,z)\|^2&=&\|h(s,y,z)-h(s,0,0)+h(s,0,0)\|^2\\\\
&\leq& \left(1+\frac{1}{\beta}\right)\|h(s,y,z)-h(s,0,0)\|^2+(1+\beta)\|h(s,0,0)\|^2\\\\
&\leq&\left(1+\frac{1}{\beta}\right)K|y|^2+(1+\beta)\|h(s,0,0)\|^2+\alpha\left(1+\frac{1}{\beta}\right)\|z\|^2,
 \end{array}$
\item []$\begin{array}{rcl}2\left<y,g(s,y)\right>&=&2\left<y,g(s,y)-g(s,0)+g(s,0)\right>\\\\
&\leq& 2\beta_2|y|^2+2|y||g(s,0)|\\\\
&\leq&(2\beta_2+1)|y|^2+|g(s,0)|^2.
 \end{array}$
\end{itemize}
Choosing $M=\frac{1-\alpha}{2},\beta=\frac{3\alpha}{1-\alpha}$ and
using Proposition \ref{pro2.1} (4), we get
\begin{eqnarray*}
 &&\mathbb{E}{\rm e}^{\lambda t+\mu A_t}|Y_t^\varepsilon|^2+\mathbb{E}\int_t^T
 \left(\lambda-2-2(\beta_1+\beta_2)-\frac{K(3-\alpha+2K)}{1-\alpha}\right)
 {\rm e}^{\lambda s+\mu A_s}|Y_s^\varepsilon|^2\,{\rm d}
s\\
&&+\mathbb{E}\int_t^T
 \left(\mu-2\beta_2-1\right){\rm e}^{\lambda s+\mu A_s}|Y_s^\varepsilon|^2\,{\rm d}A_s+\frac{1-\alpha}{6}\mathbb{E}\int_t^T{\rm
e}^{\lambda s+\mu A_s}\|Z_s^\varepsilon\|^2\,{\rm
d}s\\
&\leq&\mathbb{E}{\rm e}^{\lambda T+\mu
A_T}|\xi|^2+\mathbb{E}\int_t^T{\rm e}^{\lambda s+\mu
A_s}\left[\left(|f(s,0,0)|^2+\left(1+\frac{1+2\alpha}{1-\alpha}\right)
\|h(s,0,0)\|^2\right)\,{\rm d}s+|g(s,0)|^2\,{\rm
d}A_s\right].
\end{eqnarray*}
We show from (H3) that
\begin{equation}\label{equation3}\mathbb{E}{\rm e}^{\lambda t+\mu
A_t}|Y_t^\varepsilon|^2+\mathbb{E}\int_0^T
 {\rm e}^{\lambda t+\mu A_t}\left[|Y_t^\varepsilon|^2(\,{\rm d}
t+\,{\rm d}A_t)+\|Z_t^\varepsilon\|^2\,{\rm d}t\right]\leq C\Lambda.
\end{equation}
Therefore, the lemma follows from \eqref{equation3}
and Burkholder-Davis-Gundy's inequality.
\end{proof}
\begin{lemma}\label{lemma2}
Assume the assumptions of \ {\rm(H1)--(H4)} hold. Then, for all
$0\leq t\leq T,$ it holds that
\begin{itemize}
\item [\rm(i)] $\displaystyle \mathbb{E}\int_0^T {\rm e}^{\lambda t+\mu
A_t}\left(|\nabla \varphi_\varepsilon(Y_t^\varepsilon)|^2\,{\rm
d}t+|\nabla \psi_\varepsilon(Y_t^\varepsilon)|^2\,{\rm
d}A_t\right)\leq C\Lambda,$

\item [\rm(ii)] $\displaystyle \mathbb{E}\int_0^T {\rm e}^{\lambda t+\mu
A_t}\left(\varphi(J_\varepsilon(Y_t^\varepsilon))\,{\rm d}t+
\psi(\bar{J}_\varepsilon(Y_t^\varepsilon))\,{\rm d}A_t\right)\leq
C\Lambda,$

  \item [\rm(iii)]$\displaystyle \mathbb{E}{\rm e}^{\lambda t+\mu
A_t}\left(|Y_t^\varepsilon-J_\varepsilon(Y_t^\varepsilon)|^2+
|Y_t^\varepsilon-\bar{J}_\varepsilon(Y_t^\varepsilon)|^2\right)\leq
\varepsilon C\Lambda,$

  \item [\rm(iv)]$\displaystyle \mathbb{E}{\rm e}^{\lambda t+\mu
A_t}\left(\varphi(J_\varepsilon(Y_t^\varepsilon))+
\psi(\bar{J}_\varepsilon(Y_t^\varepsilon))\right)\leq C\Lambda.$
\end{itemize}
\end{lemma}
\begin{proof}
Here, we adopt the same arguments appeared in Pardoux and
R\u{a}\c{s}canu \cite{PR}. Given an equidistant partition of
interval $[t,T]$ such that $t=t_0<t_1<t_2<\cdots<t_n=T$ and
$t_{i+1}-t_i=\frac{1}{n}$, the subdifferential inequality shows that
\begin{eqnarray*}
{\rm e}^{\lambda t_{i+1}+\mu
A_{t_{i+1}}}\varphi_\varepsilon(Y_{t_{i+1}}^\varepsilon)\geq
\left({\rm e}^{\lambda t_{i+1}+\mu A_{t_{i+1}}}
 -{\rm e}^{\lambda t_{i}+\mu A_{t_{i}}}\right)
\varphi_\varepsilon(Y_{t_{i+1}}^\varepsilon)+{\rm e}^{\lambda
t_{i}+\mu A_{t_{i}}}\varphi_\varepsilon(Y_{t_{i}}^\varepsilon) +{\rm
e}^{\lambda t_{i}+\mu A_{t_{i}}} \left<\nabla
\varphi_\varepsilon(Y_{t_{i}}^\varepsilon),Y_{t_{i+1}}^\varepsilon-Y_{t_{i}}^\varepsilon\right>.
\end{eqnarray*}
Summing up the above formula over $i$ and letting $n\rightarrow
\infty,$ we obtain
\begin{eqnarray*}
{\rm e}^{\lambda T+\mu
A_{T}}\varphi_\varepsilon(\xi)\geq{\rm e}^{\lambda t+\mu
A_{t}}\varphi_\varepsilon(Y_t^\varepsilon)+\int_t^T {\rm e}^{\lambda
s+\mu A_{s}}\left<\nabla \varphi_\varepsilon(Y_s^\varepsilon),
\,{\rm d}Y_s^\varepsilon\right> +\int_t^T
\varphi_\varepsilon(Y_s^\varepsilon)\,{\rm d}( {\rm e}^{\lambda s+\mu
A_{s}}).
\end{eqnarray*}
From \eqref{eq1}, we obtain
\begin{eqnarray*}
&&{\rm e}^{\lambda t+\mu A_t}\left(\varphi_\varepsilon(Y_t^\varepsilon)
+\psi_\varepsilon(Y_t^\varepsilon)\right)+\int_t^T{\rm e}^{\lambda s+\mu
A_s}|\nabla \varphi_\varepsilon(Y_s^\varepsilon)|^2\,{\rm d}s+
\int_t^T{\rm e}^{\lambda s+\mu
A_s}|\nabla \psi_\varepsilon(Y_s^\varepsilon)|^2\,{\rm d}A_s\\
&&+\int_t^T{\rm e}^{\lambda s+\mu A_s}\left<\nabla
\varphi_\varepsilon(Y_s^\varepsilon),\nabla
\psi_\varepsilon(Y_s^\varepsilon)\right>(\,{\rm d}s+\,{\rm d}A_s)+
\int_t^T{\rm e}^{\lambda s+\mu A_s}\left(
\varphi_\varepsilon(Y_s^\varepsilon)+
\psi_\varepsilon(Y_s^\varepsilon)\right)(\lambda \,{\rm d}s+\mu
\,{\rm d}A_s)\\
&\leq&{\rm e}^{\lambda T+\mu
A_T}\left(\varphi_\varepsilon(\xi)+\psi_\varepsilon(\xi)\right)+
\int_t^T{\rm e}^{\lambda s+\mu A_s}\left<\nabla
\varphi_\varepsilon(Y_s^\varepsilon),f(s,Y_s^\varepsilon,Z_s^\varepsilon)\right>\,{\rm
d}s
\\&&+ \int_t^T{\rm e}^{\lambda s+\mu A_s}\left<\nabla
\psi_\varepsilon(Y_s^\varepsilon),f(s,Y_s^\varepsilon,Z_s^\varepsilon)\right>\,{\rm
d}s+ \int_t^T{\rm e}^{\lambda s+\mu A_s}\left<\nabla
\varphi_\varepsilon(Y_s^\varepsilon),g(s,Y_s^\varepsilon)\right>\,{\rm
d}A_s
\\&&+ \int_t^T{\rm e}^{\lambda s+\mu A_s}\left<\nabla
\psi_\varepsilon(Y_s^\varepsilon),g(s,Y_s^\varepsilon)\right>\,{\rm
d}A_s+ \int_t^T{\rm e}^{\lambda s+\mu A_s}\left<\nabla
\varphi_\varepsilon(Y_s^\varepsilon)+\nabla
\psi_\varepsilon(Y_s^\varepsilon),h(s,Y_s^\varepsilon,Z_s^\varepsilon)\,{\rm
d}B_s\right>
\\&&-
\int_t^T{\rm e}^{\lambda s+\mu A_s}\left<\nabla
\varphi_\varepsilon(Y_s^\varepsilon)+\nabla
\psi_\varepsilon(Y_s^\varepsilon),Z_s^\varepsilon \,{\rm
d}W_s\right>.\end{eqnarray*}
The desired results can be derived from the following facts:

 $\displaystyle\frac{1}{2 \varepsilon}|y-J_\varepsilon(y)|^2\leq
\varphi_\varepsilon(y)\leq \varphi(y),\ \frac{1}{2
\varepsilon}|y-\bar{J}_\varepsilon(y)|^2\leq \psi_\varepsilon(y)\leq
\psi(y),$

 $\displaystyle\left<\nabla \varphi_\varepsilon,f(s,y,z)\right>\leq \frac{1}{4}|\nabla
\varphi_\varepsilon|^2+12(f_t^2+K^2|y|^2+K^2\|z\|^2),$

 $\displaystyle\left<\nabla \psi_\varepsilon,g(s,y)\right>\leq \frac{1}{4}|\nabla
\psi_\varepsilon|^2+8(g_t^2+K^2|y|^2),$

 $\displaystyle\left<\nabla \psi_\varepsilon,f(s,y,z)\right>
 \leq \left<\nabla \varphi_\varepsilon,f(s,y,z)\right>^+\leq \frac{1}{4}|\nabla
\varphi_\varepsilon|^2+12(f_t^2+K^2|y|^2+K^2\|z\|^2),$

 $\displaystyle\left<\nabla \varphi_\varepsilon,g(s,y)\right>\leq
\left<\nabla \psi_\varepsilon,g(s,y)\right>^+\leq \frac{1}{4}|\nabla
\varphi_\varepsilon|^2+8(g_t^2+K^2|y|^2+K^2\|z\|^2).$
\end{proof}
\begin{lemma}\label{lemma3}
Assume the assumptions of \ {\rm(H1)--(H3)} hold. Then, it holds
that
\begin{equation}\label{equation2}
\mathbb{E}\left[\sup_{0\leq t\leq T}{\rm e}^{\lambda t+\mu
A_t}|Y_t^\varepsilon-Y_t^\delta|^2+\int_0^T {\rm e}^{\lambda t+\mu
A_t}\left(|Y_t^\varepsilon-Y_t^\delta|^2(\,{\rm d}t+\,{\rm
d}A_t)+\|Z_t^\varepsilon-Z_t^\delta\|^2\,{\rm d}t\right)\right]\leq
C(\varepsilon+\delta)\Lambda.
\end{equation}
\end{lemma}
\begin{proof}
Applying It\^{o}'s formula to ${\rm e}^{\lambda t+\mu
A_t}|Y_t^\varepsilon-Y_t^\delta|^2$, we obtain
 \begin{eqnarray*}
&& {\rm e}^{\lambda t+\mu A_t}|Y_t^\varepsilon-Y_t^\delta|^2
 +\int_t^T{\rm e}^{\lambda s+\mu A_s}|Y_s^\varepsilon-Y_s^\delta|^2\,{\rm d}(\lambda
s+\mu A_s)+\int_t^T{\rm e}^{\lambda s+\mu
A_s}\|Z_s^\varepsilon-Z_s^\delta\|^2\,{\rm d}s\\
&=&2\int_t^T{\rm e}^{\lambda s+\mu
A_s}\left<Y_s^\varepsilon-Y_s^\delta,f(s,Y_s^\varepsilon,Z_s^\varepsilon)
-f(s,Y_s^\delta,Z_s^\delta)\right>\,{\rm d}s\\&&+2\int_t^T{\rm
e}^{\lambda s+\mu
A_s}\left<Y_s^\varepsilon-Y_s^\delta,g(s,Y_s^\varepsilon)-g(s,Y_s^\delta)\right>\,{\rm
d}A_s\\
&&-2\int_t^T{\rm e}^{\lambda s+\mu
A_s}\left[\left<Y_s^\varepsilon-Y_s^\delta,\nabla
\varphi_\varepsilon(Y_s^\varepsilon)-\nabla
\varphi_\delta(Y_s^\delta) \right>\,{\rm
d}s+\left<Y_s^\varepsilon-Y_s^\delta,\nabla
\psi_\varepsilon(Y_s^\varepsilon)-\nabla \psi_\delta(Y_s^\delta)
\right>\,{\rm d}A_s\right]
\\&&+2\int_t^T{\rm e}^{\lambda s+\mu
A_s}\left<Y_s^\varepsilon-Y_s^\delta,\left(h(s,Y_s^\varepsilon,Z_s^\varepsilon)-
h(s,Y_s^\delta,Z_s^\delta)\right)\,{\rm d}B_s\right>\\&&
+\int_t^T{\rm e}^{\lambda s+\mu
A_s}||h(s,Y_s^\varepsilon,Z_s^\varepsilon)-h(s,Y_s^\delta,Z_s^\delta)||^2\,{\rm d}s\\
&&-2\int_t^T{\rm e}^{\lambda s+\mu
A_s}\left<Y_s^\varepsilon-Y_s^\delta,\left(Z_s^\varepsilon-Z_s^\delta\right)
\,{\rm d}W_s\right>.
\end{eqnarray*}
 Using the elementary inequality
$2ab\leq \beta^2 a^2+\frac{b^2}{\beta^2}$, for all $a,b\geq 0,$ and
(H1), we get
\begin{eqnarray*}
 2\left<Y_s^\varepsilon-Y_s^\delta,f(s,Y_s^\varepsilon,Z_s^\varepsilon)
-f(s,Y_s^\delta,Z_s^\delta)\right>
&=&2\left<Y_s^\varepsilon-Y_s^\delta,f(s,Y_s^\varepsilon,Z_s^\varepsilon)
-f(s,Y_s^\varepsilon,Z_s^\delta)\right.\\&&
\left.+f(s,Y_s^\varepsilon,Z_s^\delta)-f(s,Y_s^\delta,Z_s^\delta)\right>\\
&\leq& 2\beta_1|Y_s^\varepsilon-Y_s^\delta|^2+2K|Y_s^\varepsilon-Y_s^\delta|\|Z_s^\varepsilon-Z_s^\delta\|\\
&=&\left(2\beta_1+\frac{K^2}{M}\right)|Y_s^\varepsilon-Y_s^\delta|^2+M\|Z_s^\varepsilon-Z_s^\delta\|^2,
\end{eqnarray*}
\begin{eqnarray*}
\|h(s,Y_s^\varepsilon,Z_s^\varepsilon)-h(s,Y_s^\delta,Z_s^\delta)\|^2
&\leq&K|Y_s^\varepsilon-Y_s^\delta|^2+\alpha\|Z_s^\varepsilon-Z_s^\delta\|^2,
 \end{eqnarray*}
\begin{eqnarray*}
\left<Y_s^\varepsilon-Y_s^\delta,g(s,Y_s^\varepsilon)-g(s,Y_s^\delta)\right>
&\leq& 2\beta_2|Y_s^\varepsilon-Y_s^\delta|^2.
 \end{eqnarray*}\\
Choosing $M=\frac{1-\alpha}{2}$ and noting Proposition \ref{pro2.1}
(5), we get
\begin{eqnarray*}
&&\mathbb{E}{\rm e}^{\lambda t+\mu A_t}|Y_t^\varepsilon|^2+\mathbb{E}\int_t^T
 \left(\lambda-2\beta_2-\frac{2K^2}{1-\alpha}-K\right){\rm e}^{\lambda s+\mu A_s}|Y_s^\varepsilon|^2\,{\rm d}s\\
&&+\mathbb{E}\int_t^T \left(\mu-2\beta_2\right){\rm e}^{\lambda s+\mu A_s}|Y_s^\varepsilon|^2\,{\rm d} A_s
 +\frac{1-\alpha}{2}\mathbb{E}\int_t^T{\rm e}^{\lambda s+\mu A_s}\|Z_s^\varepsilon\|^2\,{\rm d}s\\
&\leq& 2(\varepsilon+\delta)\mathbb{E}\int_t^T{\rm e}^{\lambda s+\mu A_s}\left[\left<\nabla
\varphi_\varepsilon(Y_s^\varepsilon),\nabla\varphi_\delta(Y_s^\delta)\right>\,{\rm
d}s +\left<\nabla\psi_\varepsilon(Y_s^\varepsilon),
\nabla\psi_\delta(Y_s^\delta)\right>\,{\rm
d}A_s\right].
\end{eqnarray*}
Thus, the desired result follows from
Lemma \ref{lemma2} (i) and Burkholder-Davis-Gundy's
inequality.\end{proof}

We now give the following:
\begin{proof}[Proof of Theorem \ref{thm3.1}] {\it Existence} \newline Lemma
\ref{lemma3} shows that $(Y^\varepsilon,Z^\varepsilon)$ is a Cauchy
sequence in $
 \left(\mathcal{S}_k^{\lambda,\mu} \cap\mathcal{M}_k^{\lambda,\mu}
 \cap  \bar{\mathcal{M}}_k^{\lambda,\mu}\right)\times
 \mathcal{M}_{k\times d}^{\lambda,\mu}.$ We denote its limit as
$(Y,Z)$. Then, $(Y,Z) \in \left(\mathcal{S}_k^{\lambda,\mu}
\cap\mathcal{M}_k^{\lambda,\mu}
 \cap  \bar{\mathcal{M}}_k^{\lambda,\mu}\right)\times
 \mathcal{M}_{k\times d}^{\lambda,\mu}.$ Lemma
\ref{lemma2} shows that
$$
\lim_{\varepsilon\rightarrow 0}J_\varepsilon(Y^\varepsilon)=Y \
\mbox{in}\ \mathcal{M}_k^{\lambda,\mu}, \
\lim_{\varepsilon\rightarrow 0}\bar{J}_\varepsilon(Y^\varepsilon)=Y
\ \mbox{in}\ \bar{\mathcal{M}}_k^{\lambda,\mu}$$ and for all $0\leq
t\leq T,$
$$
\lim_{\varepsilon\rightarrow 0}\mathbb{E}{\rm e}^{\lambda t+\mu
A_t}\left[|J_\varepsilon(Y_t^\varepsilon)-Y_t|^2+
|\bar{J}_\varepsilon(Y_t^\varepsilon)-Y_t|^2\right]=0.$$ Fatou's
lemma, Lemma \ref{lemma2} and the fact that $\varphi$ and $\psi$ are
l.s.c. show that (3) of Definition \ref{def2.1} is satisfied. In
addition, from Lemma \ref{lemma2}, we have
 $$\mathbb{E} \int_0^T{\rm e}^{\lambda t+\mu
A_t}\left(|U_t^\varepsilon|^2\,{\rm d}t+|V_t^\varepsilon|^2\,{\rm
d}A_t\right)\leq C \Lambda,$$
 which shows that $U_t^\varepsilon $ and $V_t^\varepsilon$ are bounded
in the space $\mathcal{M}_{k}^{\lambda,\mu}$ and
$\bar{\mathcal{M}}_{k}^{\lambda,\mu}$ respectively. So, there exists
a subsequence $\varepsilon_n \to 0$ such that
$$U^{\varepsilon_n}\rightharpoonup U,\ \mbox{weakly in the space} \
\mathcal{M}_{k}^{\lambda,\mu},$$
$$V^{\varepsilon_n}\rightharpoonup V,\ \mbox{weakly in the space} \
\bar{\mathcal{M}}_{k}^{\lambda,\mu}.$$
Furthermore, we get
$$\mathbb{E} \int_0^T{\rm e}^{\lambda t+\mu
A_t}\left(|U_t|^2\,{\rm d}t+|V_t|^2\,{\rm
d}A_t\right)\leq \liminf_{n\to \infty} \mathbb{E} \int_0^T{\rm
e}^{\lambda t+\mu A_t}\left(|U_t^{\varepsilon_n}|^2\,{\rm
d}t+|V_t^{\varepsilon_n}|^2\,{\rm d}A_t\right)\leq C \Lambda.$$
Thus, the process $(Y,U,V,Z)$ satisfies $(5)$ of Definition
\ref{def2.1} by passing limit in \eqref{eq1}.

Finally, we show that $(4)$ of Definition \ref{def2.1} is satisfied.
For all $0\leq t\leq T$, since  $U_t^\varepsilon\in \partial
\varphi(J_\varepsilon(Y_t^\varepsilon))$ and $V_t^\varepsilon\in
\partial\psi(\bar{J}_\varepsilon(Y_t^\varepsilon))$, it follows that
\begin{eqnarray*}
{\rm e}^{\lambda t+\mu
A_t}\left<U_t^\varepsilon,U_t-J_\varepsilon(Y_t^\varepsilon)\right>+{\rm
e}^{\lambda t+\mu A_t}\varphi(J_\varepsilon(Y_t^\varepsilon))\leq
{\rm e}^{\lambda t+\mu A_t}\varphi(U_t),\ \,{\rm d}\mathbb{P}\times
\,{\rm d}t\mbox{-a.e.},
\end{eqnarray*}
and
\begin{eqnarray*}
{\rm e}^{\lambda t+\mu
A_t}\left<V_t^\varepsilon,V_t-\bar{J}_\varepsilon(Y_t^\varepsilon)\right>+{\rm
e}^{\lambda t+\mu A_t}\psi(\bar{J}_\varepsilon(Y_t^\varepsilon))\leq
{\rm e}^{\lambda t+\mu A_t}\varphi(V_t),\ \,{\rm d}\mathbb{P}\times
A(\omega,\,{\rm d}t)\mbox{-a.e.}
\end{eqnarray*}
Taking the $\liminf$ in the above two inequalities, $(4)$ of Definition \ref{def2.1} holds.

{\it Uniqueness}\newline Let $(Y_t,U_t,V_t,Z_t)_{0\leq t \leq T}$ and
$(Y_t^\prime,U_t^\prime,V_t^\prime,Z_t^\prime)_{0\leq t \leq T}$ be
two solutions of the BDSGVI \eqref{equation1}. Denote
$$(\triangle Y_t, \triangle U_t,\triangle V_t,\triangle Z_t)_{0\leq t \leq
T}=(Y_t-Y_t^\prime,U_t-U_t^\prime,V_t-V_t^\prime,Z_t-Z_t^\prime)_{0\leq
t \leq T}.$$Applying It\^{o}'s  formula to ${\rm e}^{\lambda t+\mu
A_t}|\triangle Y_t|^2$ shows that

 $\displaystyle\mathbb{E}{\rm e}^{\lambda t+\mu
A_t}|\triangle Y_t|^2+2\mathbb{E}\int_t^T{\rm e}^{\lambda s+\mu
A_s}\left<\triangle U_s, \triangle Y_s\right>\,{\rm
d}s+2\mathbb{E}\int_t^T{\rm e}^{\lambda s+\mu A_s}\left<\triangle
V_s, \triangle Y_s\right>\,{\rm d}A_s$
$$+\mathbb{E}\int_t^T{\rm e}^{\lambda s+\mu
A_s}|\triangle Y_s|^2(\lambda \,{\rm d}s+\mu \,{\rm d}A_s)
 + \mathbb{E}\int_t^T{\rm e}^{\lambda s+\mu
A_s}\|\triangle Z_s\|^2\,{\rm d}s$$
\begin{eqnarray}
&=&2\mathbb{E}\int_t^T{\rm e}^{\lambda s+\mu
A_s}\left<\triangle Y_s,f(s,Y_s,Z_s)-f(s,Y_s^\prime,Z_s^\prime)\right>\,{\rm d}s\nonumber\\
&&+2\mathbb{E}\int_t^T{\rm e}^{\lambda s+\mu A_s}\left<\triangle
Y_s,g(s,Y_s)-g(s,Y_s^\prime)\right>\,{\rm
d}A_s\nonumber\\&&+\mathbb{E}\int_t^T {\rm e}^{\lambda s+\mu
A_s}\|h(s,Y_s,Z_s)-h(s,Y_s^\prime,Z_s^\prime)\|^2\,{\rm d}s.
 \end{eqnarray}
Since $\partial \varphi$ and $\partial \psi$ are monotone, we obtain
$$\left<\triangle U_s, \triangle Y_s\right>\geq 0,\ \,{\rm d}\mathbb{P}\times {\rm
d}t\mbox{-a.e.},\ \left<\triangle V_s, \triangle Y_s\right>\geq 0,\
\,{\rm d}\mathbb{P}\times A(\omega,{\rm d}t)\mbox{-a.e.}$$ Thus, as
the same procedure as Lemma \ref{lemma3}, we can show the uniqueness
of the solution.
\end{proof}

\section{Stochastic viscosity solutions of SPVI with a mixed nonlinear Neumann-Dirichlet boundary condition}

In this section, we consider the one-dimensional equation, i.e. $k =
1$. We will investigate the BDSGVI studied in the previous section
in order to give the existence of the stochastic viscosity solution
of a class of SPVI with a mixed nonlinear Neumann-Dirichlet boundary
condition. For this, we need some additional hypotheses and tools.

\subsection{Notion of stochastic viscosity solution of SPVI with a
mixed nonlinear Neumann-Dirichlet boundary condition}

With the same notations as in Section 2, let ${\bf
F}^{B}=\{\mathcal{F}_{t,T}^{B}\}_{0\leq t\leq T}$ be the filtration
generated by $B$, where $B$ is a one-dimensional Brownian motion. By
${\mathcal{M}}^{B}_{0,T}$, we denote all the ${\bf F}^{B}$-stopping
times $\tau$ such $0\leq \tau\leq T$, a.s.
${\mathcal{M}}^{B}_{\infty}$ is the set of all almost surely finite
${\bf F}^{B}$-stopping times. For generic Euclidean spaces $E$ and
$E_{1}$, we introduce the following spaces:
\begin{enumerate}
\item The symbol $\mathcal{C}^{k,n}([0,T]\times
E; E_{1})$ stands for the space of all $E_{1}$-valued functions
defined on $[0,T]\times E$ which are $k$-times continuously
differentiable in $t$ and $n$-times continuously differentiable in
$x$, and $\mathcal{C}^{k,n}_{b}([0,T]\times E; E_{1})$ denotes the
subspace of $\mathcal{C}^{k,n}([0,T]\times E; E_{1})$ in which all
functions have uniformly bounded partial derivatives.
\item For any sub-$\sigma$-field $\mathcal{G} \subseteq
\mathcal{F}_{T}^{B}$, $\mathcal{C}^{k,n}(\mathcal{G},[0,T]\times E;
E_{1})$ (resp.\, $\mathcal{C}^{k,n}_{b}(\mathcal{G},[0,T]\times E;
E_{1})$) denotes the space of all $\mathcal{C}^{k,n}([0,T]\times E;
E_{1})$  (resp.\, $\mathcal{C}^{k,n}_{b}([0,T]\times E;E_{1})$-valued
random variable that are $\mathcal{G}\otimes\mathcal{B}([0,T]\times
E)$-measurable;
\item $\mathcal{C}^{k,n}({\bf F}^{B},[0,T]\times E; E_{1})$
(resp.$\mathcal{C}^{k,n}_{b}({\bf F}^{B},[0,T]\times E; E_{1})$) is
the space of all random fields $\phi\in
\mathcal{C}^{k,n}({\mathcal{F}}_{T},[0,T]\times E; E_{1}$ (resp.
$\mathcal{C}^{k,n}({\mathcal{F}}_{T},[0,T]\times E; E_{1})$, such
that for fixed $x\in E$ and $t\in [0,T]$, the mapping
$\displaystyle{\omega\rightarrow \alpha(t,\omega,x)}$ is
${\bf F}^{B}$-progressively measurable.
\item For any sub-$\sigma$-field $\mathcal{G} \subseteq
\mathcal{F}^{B}$ and a real number $ p\geq 0$, let
$L^{p}(\mathcal{G};E)$ be a set of all $E$-valued,
$\mathcal{G}$-measurable random variable $\xi$ such that $
\E|\xi|^{p}<\infty$.
\end{enumerate}
Furthermore, regardless of the dimension, we denote by
$\left<\cdot,\cdot\right>$ and $|\cdot|$ the inner product and norm
in $E$ and $E_1$, respectively. For
$(t,x,y)\in[0,T]\times\R^{d}\times\R$, we denote
$D_{x}=(\frac{\partial}{\partial
x_{1}},....,\frac{\partial}{\partial x_{d}}),\,
D_{xx}=(\partial^{2}_{x_{i}x_{j}})_{i,j=1}^{d}$,
$D_{y}=\frac{\partial}{\partial y}, \,\
D_{t}=\frac{\partial}{\partial t}$. The meaning of $D_{xy}$ and
$D_{yy}$ is then self-explanatory.

Let $\Theta$ be an open connected and smooth bounded domain of $\R^{d}\, (d\geq
1)$ such that
for a function $\phi\in\mathcal{C}^{2}_b(\R^{d}),\ \Theta$ and its
boundary  $\partial\Theta$ are characterized by
$\Theta=\{\phi>0\},\, \partial\Theta=\{\phi=0\}$ and, for any
$x\in\partial\Theta,\, \nabla \phi(x)$ is the unit normal vector
pointing towards the interior of $\Theta$.
In this section, we consider the continuous coefficients $b,\,\sigma,\, l,\, f,\, \phi$ and $h$
\begin{eqnarray*}
f&:&\Omega\times[0,T]\times\overline{\Theta}\times\R\times\R^{d}\to
\R\\
g&:&\Omega\times[0,T]\times\overline{\Theta}\times\R\to
\R\\
\sigma&:&\R^{d}\to\R^{d\times
d} \; \mbox{and}\,\, b:\R^{d}\to\R^{d}\\
\chi&:&\overline{\Theta}\to \R,
\end{eqnarray*}
satisfy that
\begin{description}
\item $
(\rm{ H5})\ \left\{
\begin{array}{l}
|f(t,x,y,z)|\leq K(1+|x|+|y|+\|z\|),\\\\
|g(t,x,y)|\leq K(1+|x|+|y|),\\\\
|\chi(x)|+|\varphi(\chi(x))|+|\psi(\chi(x))|\leq K(1+|x|).
\end{array}\right.
$
\item
$ (\rm{ H6})\ \left\{
\begin{array}{l}
|b(x)-b(x')|+\|\sigma(x)-\sigma(x')\|\leq c|x-x'|,\\\\
|f(t,x,y,z)-f(t,x,y',z')|\leq c(|y-y'|+\|z-z'\|),\\\\
\langle y-y',g(t,x,y)-g(t,x,y')\rangle\leq \beta|y-y'|^2.
\end{array}\right.
$
\item $(\rm{H7})$ The function
$h\in{\mathcal{C}}_{b}^{0,2,3}([0,T]\times\overline{\Theta}\times\R;\R)$.
\end{description}
Let us consider the following SPVI with mixed nonlinear
Neumann-Dirichlet boundary condition:
\begin{eqnarray}
\mathcal{SPVI}^{(f,g,h,\chi,\varphi,\psi)}\left\{
\begin{array}{l}
{\rm(i)}\;\displaystyle\Bigg(\frac{\partial u(t,x)}{\partial
t}+\left[
Lu(t,x)+f(t,x,u(t,x),\sigma^{*}(x)D_{x}u(t,x))\right]\\\\
~~~~~~~~~~~~~~~~-h(t,x,u(t,x))\dot{B}_{s}\Bigg)\in\partial\varphi,\,\,\
(t,x)\in[0,T]\times\Theta,\\\\
{\rm (ii)}\;\displaystyle\frac{\partial u}{\partial
n}(t,x)+g(t,x,u(t,x))\in\partial\psi,\,\,\ (t,x)\in[0,T]\times\partial\Theta, \\\\
{\rm (iii)}\; u(T,x)=\chi(x),\,\,\,\,\,\,\ x\in\overline{\Theta},
\end{array}\right.\label{SPVI}
\end{eqnarray}
where
\begin{eqnarray}
L=\frac{1}{2}\sum_{i,j=1}^{d}(\sigma(x)\sigma^{*}(x))_{i,j}
\frac{\partial^{2}}{\partial x_{i}\partial
x_{j}}+\sum^{d}_{i=1}b_{i}(x)\frac{\partial}{\partial x_{i}},\quad
\forall\, x\in\Theta,\label{operateur}
\end{eqnarray}
and
\begin{eqnarray*}
\frac{\partial}{\partial n}=\sum_{i=1}^{d}\frac{\partial\phi }{\partial
x_{i}}(x)\frac{\partial}{\partial x_{i}},\quad \forall\,
x\in\partial\Theta.
\end{eqnarray*}
We first define the mean of stochastic viscosity solution to
$\mathcal{SPVI}^{(f,g,h,\chi,\varphi,\psi)}$ \eqref{SPVI}. As
mentioned in introduction, our notion of stochastic viscosity
solution uses the stochastic sub- and super-jets introduced by
Buckdahn and Ma \cite{BMa}. Next, the existence result will be
derived by the use of the  well-known Doss-Sussman transformation.
In this fact, let us recall the statement appeared in \cite{BMa}.
\begin{definition}
Let $\tau\in {\mathcal{M}}^{B}_{0,T}$, and
$\xi\in\mathcal{F}_{\tau}$. We say that a sequence of random
variables $(\tau_k,\xi_k)$ is a $(\tau,\xi)$-approximating sequence
if for all $k$, $(\tau_k,\xi_k)\in{\mathcal{M}}^{B}_{\infty}\times
L^{2}(\mathcal{F}_{\tau},\Theta)$ such that
\begin{itemize}
\item [(i)] $\xi_k\rightarrow\xi$  in probability;
\item [(ii)] either $\tau_k\uparrow\tau$ a.s., and $\tau_k<\tau$ on the set $\{\tau>0\}$; or $\tau_k\downarrow\tau$ a.s., and $\tau_k>\tau$ on the set $\{\tau<T\}$.
\end{itemize}
\end{definition}
\begin{definition}
Let $(\tau,\xi)\in \mathcal{M}_{0,T}^B\times
L^2\left(\mathcal{F}^{B}_{\tau}; \Theta\right)$ and $u\in
\mathcal{C}\left(\mathbf{F}^B, [0,T]\times
\overline{\Theta}\right)$. A triplet of $(a, p,X)$ is called a
stochastic $h$-superjet of $u$ at $(\tau,\xi)$ if the following
terms hold:
\begin{itemize}
\item [(i)] $(a,b,c,p,q,X)$ is an $\R\times\R\times\R\times\R^d\times\R^n\times\mathcal{S}(n)$-valued,
 $\mathcal{F}_{\tau}^B$-measurable random
vector, where $\mathcal{S}(n)$ is the set of all symmetric $n\times n$ matrix.
\item[(ii)] Denoting
\begin{eqnarray*}
\left\{
\begin{array}{l}
b=g(\tau,\xi,u(\tau,\xi)),\;\;\; c=(g\partial_ug)(\tau,\xi,u(\tau,\xi))\\\\
q=\partial_xg(\tau,\xi,u(\tau,\xi))+\partial_ug(\tau,\xi,u(\tau,\xi))p.
\end{array}\right.
\end{eqnarray*}
\end{itemize}
Then, for all $(\tau, \xi)$-approximating sequence $(\tau_k,\xi_k)$,
it holds that
\begin{eqnarray}
u(\tau_k,\xi_k)&\leq& u(\tau,\xi)+a(\tau_k-\tau)+b(B_{\tau_k}-B_{\tau})+\frac{c}{2}(B_{\tau_k}-B_{\tau})^2+\langle p,\xi_k-\xi\rangle\nonumber\\
&&+\langle q,\xi_k-\xi\rangle(B_{\tau_k}-B_{\tau})+\frac{1}{2}\langle X(\xi_k-\xi),\xi_k-\xi\rangle\nonumber\\
&&+o(|\tau_k-\tau|)+o(|\xi_k-\xi|^2).\label{jet}
\end{eqnarray}
Next, $\mathcal{J}^{1,2,+}_{h} u(\tau,\xi)$ denotes the set of all
stochastic $h$-superjet of $u$ at $(\tau, \xi)$. Similarly, the
triplet of $(a, p,X)$ is a stochastic $h$-subjet of $u$ at $(\tau,
\xi)$ if $(i)$ and $(ii)$ hold and the inequality in \eqref{jet} is
reversed. Moreover, $\mathcal{J}^{1,2,-}_{h} u(\tau,\xi)$ denotes
the set of all stochastic $h$-subjet of $u$ at $(\tau,\xi)$.
\end{definition}
\begin{remark}
For $\theta$ equal to $\varphi$ or $\psi$, we emphasize that $\partial\theta(y)= [\theta'_l(y),\theta'_r(y)]$, for every $y\in {\rm
Dom}(\theta)$ where $\theta'_l(y)$ and $\theta'_r(y)$ denote the left and right derivatives of $\theta$.
\end{remark}
Now, we define the stochastic viscosity solution of
$\mathcal{SPVI}^{(f,g,h,\chi,\varphi,\psi)}$ \eqref{SPVI}. In order
to simplify the presentation, we set
\begin{eqnarray*}
V_{f}(\tau,\xi,a,p,X)=-a-\frac{1}{2}{\rm Trace}(\sigma\sigma^*(\xi)X)-\langle
p,b(\xi)\rangle-f\left(\tau,\xi,u(\tau,\xi),\sigma^*(\xi)p\right).
\end{eqnarray*}
\begin{definition}\label{defvisco}
A random field $u \in \mathcal{C}\left(\mathbf{F}^B, [0,T]\times
\overline{\Theta}\right)$   which satisfies that
$u\left(T,x\right)=\chi\left(x\right)$, for all $x\in
\overline{\Theta}$, is called a stochastic viscosity subsolution of
$\mathcal{SPVI}^{(f,g,h,\chi,\varphi,\psi)}$ \eqref{SPVI} if
\begin{eqnarray*}
u(\tau,\xi)&\in & {\rm Dom}(\varphi),\;\;\;\;\;\; \forall\; (\tau,\xi)\in\mathcal{M}_{0,T}^B\times L^2\left(\mathcal{F}^{B}_{\tau};\Theta\right),\;\;\; \P\mbox{-a.s.},\\
u(\tau,\xi)&\in & {\rm Dom}(\psi),\;\;\;\;\;\; \forall\;
(\tau,\xi)\in\mathcal{M}_{0,T}^B\times
L^2\left(\mathcal{F}^{B}_{\tau};\partial\Theta\right),\;\;\;
\P\mbox{-a.s.},
\end{eqnarray*}
and at any $(\tau,\xi)\in\mathcal{M}_{0,T}^B\times L^2\left(\mathcal{F}^{B}_{\tau};\Theta\right)$, for any $(a, p,X)\in\mathcal{J}^{1,2,+}_{h} u(\tau,\xi)$, it hold $\P$-a.s.
\begin{itemize}
\item[(a)] on the event $\left\{0<\tau<T\right\}\cap\left\{\xi\in
\Theta\right\}$
\begin{equation}\label{E:def1}
V_{f}(\tau,\xi,a,p,X)+\varphi'_l(u(\tau,\xi)-\frac{1}{2}(h\partial_uh)(\tau,\xi,u(\tau,\xi))\leq 0;
\end{equation}
\item[(b)] on the event $\left\{0<\tau<T\right\}\cap\left\{\xi\in
\partial \Theta\right\}$

$\displaystyle\min\Big(V_{f}(\tau,\xi,a,p,X)+\varphi'_l(u(\tau,\xi)-\frac{1}{2}(h\partial_uh)(\tau,\xi,u(\tau,\xi)),$
\begin{align}
\langle\nabla
\phi(\xi),p\rangle-g(\tau,\xi,u(\tau,\xi))+\psi'_l(u(\tau,\xi)\Big)\leq
0. \label{E:viscosity01}
\end{align}
\end{itemize}
A random field $u \in \mathcal{C}\left(\mathbf{F}^B, [0,T]\times
\overline{\Theta}\right)$   which satisfies that
$u\left(T,x\right)=\chi\left(x\right)$, for all $x\in
\overline{\Theta}$, is called a stochastic viscosity supersolution
of $\mathcal{SPVI}^{(f,g,h,\chi,\varphi,\psi)}$ \eqref{SPVI} if
\begin{eqnarray*}
u(\tau,\xi)&\in & {\rm Dom}(\varphi),\;\;\;\;\;\; \forall\;(\tau,\xi)\in\mathcal{M}_{0,T}^B\times L^2\left(\mathcal{F}^{B}_{\tau};\Theta\right),\;\;\; \P\mbox{-a.s.},\\
u(\tau,\xi)&\in & {\rm Dom}(\psi),\;\;\;\;\;\;
\forall\;(\tau,\xi)\in\mathcal{M}_{0,T}^B\times
L^2\left(\mathcal{F}^{B}_{\tau};\partial\Theta\right),\;\;\;
\P\mbox{-a.s.},
\end{eqnarray*}
and at any $(\tau,\xi)\in\mathcal{M}_{0,T}^B\times L^2\left(\mathcal{F}^{B}_{\tau};\Theta\right)$, for any $(a, p,X)\in\mathcal{J}^{1,2,-}_{h} u(\tau,\xi)$, it hold $\P$-a.s.
\begin{itemize}
\item[(a)] on the event $\left\{0<\tau<T\right\}\cap\left\{\xi\in
\Theta\right\}$
\begin{equation}\label{E:def01}
V_f(\tau,\xi,a,p,X)+\varphi'_r(u(\tau,\xi)-\frac{1}{2}(g\partial_ug)(\tau,\xi,u(\tau,\xi))\geq 0;
\end{equation}
\item[(b)] on the event $\left\{0<\tau<T\right\}\cap\left\{\xi\in
\partial \Theta\right\}$
\begin{align}
\max\left(V_f(\tau,\xi,a,p,X)+\varphi'_r(u(\tau,\xi)-\frac{1}{2}(h\partial_uh)(\tau,\xi,u(\tau,\xi)),\langle\nabla \phi(\xi),p\rangle-g(\tau,\xi,u(\tau,\xi))\right)\geq 0. \label{E:viscosity001}
\end{align}
\end{itemize}
Finally, a random field $u \in \mathcal{C}\left(\mathbf{F}^B,
[0,T]\times \overline{\Theta}\right)$ is called a stochastic
viscosity solution of $\mathcal{SPVI}^{(f,g,h,\chi,\varphi,\psi)}$
\eqref{SPVI} if it is both a stochastic viscosity subsolution and a
stochastic viscosity supersolution.
\end{definition}
\begin{remark}
Observe that if $f$ and $g$ are deterministic and $h\equiv 0$,
Definition\, $\ref{defvisco}$ coincides with the definition of
(deterministic) viscosity solution of PVI given by Maticiuc and R\u{a}\c{s}canu in \cite{MR}.
\end{remark}
To end this section, we state the notion of random viscosity solution which will be a bridge
link to the stochastic viscosity solution and its deterministic counterpart.
\begin{definition}
A random field $u\in C({\bf F}^B, [0,T]\times\R^n)$ is called an $\omega$-wise viscosity solution if
for $\P$-almost all $\omega\in \Omega,\;  u(\omega,\cdot,\cdot)$ is a (deterministic) viscosity solution of $\mathcal{SPVI}^{(f,g,0,\chi,\varphi,\psi)}$.
\end{definition}

\subsection{Doss-Sussmann transformation}
In this section, using the Doss-Sussman transformation, our next goal
is to establish the existence of the stochastic viscosity solution
to $\mathcal{SPVI}^{(f,g,h,\chi,\varphi,\psi)}$ \eqref{SPVI}  by
means of backward doubly stochastic generalized variational
inequality.

As shown by the work of Buckdahn and Ma \cite{BM1,BM2}, the Doss transformation will depend heavily on
the following stochastic flow $\eta\in C({\bf F}^B, [0,
T]\times\R^n\times\R)$, defined as the unique solution of the
following stochastic differential equation in the Stratonovich
sense:
\begin{eqnarray}
\eta(t,x,y)&=&y+\int_t^T\langle h(s,x,\eta(s,x,y)),
\circ dB_s\rangle.\label{p1}
\end{eqnarray}
We refer the reader to their paper \cite{BM1} for a lucid discussion
on this topic. We also note that due the direction of backward Itô
integral, \eqref{p1} should be viewed as going from $T$ to $t$ (i.e
$y$ should be understood as the initial value). Under the assumption
$({H7})$, the mapping $y\mapsto \eta(t,x,y)$
 defines a diffeomorphism for all $(t,x),\; \P$-a.s. (see Protter \cite{Pr}). Let us denote its $y$-inverse
  by $\varepsilon(t,x,y)$. Then, one can show that $\varepsilon(t, x, y)$ is the solution to the following
  first-order SPDE:
\begin{eqnarray*}
\varepsilon(t,x,y)=y-\int_t^T\langle D_y\varepsilon(s,x,y),\, h(s,x,\eta(s,x,y))
\circ dB_s\rangle.
\end{eqnarray*}
Let us recall the following important proposition appeared in
\cite{BM} (see Lemma 4.8).
\begin{proposition}
Assume that the assumptions $({H1})$--$({H7})$ hold. Let
$(\tau,\xi)\in \mathcal{M}_{0,T}^B\times
L^2\left(\mathcal{F}^{B}_{\tau}; \Theta\right)$, $u\in
\mathcal{C}\left(\mathbf{F}^B, [0,T]\times\overline{\Theta}\right)$
and $(a_u,X_u,p_u)\in\mathcal{J}^{1,2,+}_{h} u(\tau,\xi)$. Define
$v(\cdot,\cdot) = \varepsilon(\cdot,\cdot, u(\cdot,\cdot))$. Then,
for any $(\tau,\xi)$-approximating sequence $(\tau_k,\xi_k)$, and
for $\P$-a.e. $\omega$, it holds that
\begin{eqnarray*}
v(\tau_k,\xi_k)&\leq& v(\tau,\xi)+a_v(\tau_k-\tau)+b_v(B_{\tau_k}-B_{\tau})+\langle p_v,\xi_k-\xi\rangle\nonumber\\
&&+\langle q_v,\xi_k-\xi\rangle(B_{\tau_k}-B_{\tau})+\frac{1}{2}\langle X_v(\xi_k-\xi),\xi_k-\xi\rangle\nonumber\\
&&+o(|\tau_k-\tau|)+o(|\xi_k-\xi|^2).\label{jet}
\end{eqnarray*}
where
\begin{eqnarray*}
\left\{
\begin{array}{l}
a_v=D_y\varepsilon(\tau,\xi,u(\tau,\xi))a_u\\\\
p_v=D_y\varepsilon(\tau,\xi,u(\tau,\xi))p_u+D_x\varepsilon(\tau,\xi,u(\tau,\xi))\\\\
X_v=D_y\varepsilon(\tau,\xi,u(\tau,\xi))X_u+2D_{xy}\varepsilon(\tau,\xi,u(\tau,\xi))p^*_u+D_{xx}\varepsilon(\tau,\xi,u(\tau,\xi))
+D_{yy}\varepsilon(\tau,\xi,u(\tau,\xi))p_up_u^*
\end{array}\right.
\end{eqnarray*}
Namely, $(a_v,X_v,p_v)\in\mathcal{J}^{1,2,+}_{0} v(\tau,\xi)$

Conversely, let $(\tau,\xi)\in \mathcal{M}_{0,T}^B\times
L^2\left(\mathcal{F}^{B}_{\tau}; \Theta\right)$, $v\in
\mathcal{C}\left(\mathbf{F}^B, [0,T]\times\overline{\Theta}\right)$
and $(a_v,X_v,p_v)\in\mathcal{J}^{1,2,+}_{0} v(\tau,\xi)$. Define
$u(\cdot,\cdot) = \eta(\cdot,\cdot, v(\cdot,\cdot))$. Then, the
triplet $(a_u,X_u,p_u)$ given by
\begin{eqnarray*}
\left\{
\begin{array}{l}
a_u=D_y\eta(\tau,\xi,v(\tau,\xi))a_v\\\\
p_u=D_y\eta(\tau,\xi,v(\tau,\xi))p_v+D_x\eta(\tau,\xi,v(\tau,\xi))\\\\
X_u=D_y\eta(\tau,\xi,v(\tau,\xi))X_v+2D_{xy}\eta(\tau,\xi,v(\tau,\xi))p^*_v+D_{xx}\eta(\tau,\xi,v(\tau,\xi))
+D_{yy}\eta(\tau,\xi,v(\tau,\xi))p_vp_v^*
\end{array}\right.
\end{eqnarray*}
satisfies $(a_u,X_u,p_u)\in\mathcal{J}^{1,2,+}_{h} u(\tau,\xi)$.
\end{proposition}
Following the key ideas of Buckdahn and Ma, our aim is  to convert a
SPVI to a PVI with random coefficients with the Doss-Sussman
transformation so that the stochastic viscosity solution can be
studied $\omega$-wisely. However, our resulting equation from
$\mathcal{SPVI}^{(f,g,h,\chi,\varphi,\psi)}$ \eqref{SPVI} due to
Doss-Sussman transformation is not necessarily the PVI studied by
Maticiuc and R\u{a}\c{s}canu in \cite{MR}. Therefore, we will need
the following version of Doss-Sussman transformation.
\begin{corollary}\label{corollary4.8}
Assume that the assumptions $({H1})$--$({H7})$ hold. Let
$(\tau,\xi)\in \mathcal{M}_{0,T}^B\times
L^2\left(\mathcal{F}^{B}_{\tau}; \Theta\right)$, $u\in
\mathcal{C}\left(\mathbf{F}^B, [0,T]\times\overline{\Theta}\right)$
and  define $v(\cdot,\cdot) = \varepsilon(\cdot,\cdot,
u(\cdot,\cdot))$.
\begin{description}
\item If $(a_u,X_u,p_u)\in\mathcal{J}^{1,2,+}_{h}
u(\tau,\xi)$, then $u$ satisfies \eqref{E:def1} and
\eqref{E:viscosity01} if and only if $v(\cdot,\cdot)$ satisfies that
\begin{itemize}
\item[\rm(a)] on the event $\left\{0<\tau<T\right\}\cap\left\{\xi\in
\Theta\right\}$
\begin{equation}\label{E:def2}
V_{\widetilde{f}}(\tau,\xi,a_v,p_v,X_v)+\frac{\varphi'_l(\eta(\tau,\xi,v(\tau,\xi))}{D_y\eta(\tau,\xi,v(\tau,\xi))}\leq 0;
\end{equation}
\item[\rm(b)] on the event $\left\{0<\tau<T\right\}\cap\left\{\xi\in
\partial \Theta\right\}$

$\displaystyle
\min\Bigg(V_{\widetilde{f}}(\tau,\xi,a_v,p_v,X_v)+\frac{\varphi'_l
(\eta(\tau,\xi,v(\tau,\xi))}{D_y\eta(\tau,\xi,v(\tau,\xi))},$
\begin{align}
\langle\nabla
\phi(\xi),p_v\rangle-\widetilde{g}(\tau,\xi,u(\tau,\xi))+\frac{\psi'_l
(\eta(\tau,\xi,v(\tau,\xi))}{D_y\eta(\tau,\xi,v(\tau,\xi))}\Bigg)\leq
0, \label{E:viscosity02}
\end{align}
\end{itemize}
\item where $(a_v, p_v,X_v)$ appear in Proposition 4.7 and functions $\widetilde{f}$ and $\widetilde{g}$
 will be defined in the proof.
\item If $(a_u,X_u,p_u)\in\mathcal{J}^{1,2,-}_{h} u(\tau,\xi)$, then $u$
 satisfies \eqref{E:def01} and \eqref{E:viscosity001} if and only if $v(\cdot,\cdot)$
 satisfies that
\begin{itemize}
\item[\rm(a)] on the event $\left\{0<\tau<T\right\}\cap\left\{\xi\in
\Theta\right\}$
\begin{equation}\label{E:def02}
V_{\widetilde{f}}(\tau,\xi,a_v,p_v,X_v)+\frac{\varphi'_l(\eta(\tau,\xi,v(\tau,\xi))}{D_y\eta(\tau,\xi,v(\tau,\xi))}\geq 0;
\end{equation}
\item[\rm(b)] on the event $\left\{0<\tau<T\right\}\cap\left\{\xi\in
\partial \Theta\right\}$

$\displaystyle
\max\Bigg(V_{\widetilde{f}}(\tau,\xi,a_v,p_v,X_v)+\frac{\varphi'_l
(\eta(\tau,\xi,v(\tau,\xi))}{D_y\eta(\tau,\xi,v(\tau,\xi))},$
\begin{align}
\langle\nabla
\phi(\xi),p_v\rangle-\widetilde{g}(\tau,\xi,u(\tau,\xi))
+\frac{\psi'_l(\eta(\tau,\xi,v(\tau,\xi))}{D_y\eta(\tau,\xi,v(\tau,\xi))}\Bigg)\geq
0, \label{E:viscosity002}
\end{align}
\end{itemize}
\end{description}
\end{corollary}
\begin{proof}
Let $(\tau,\xi)\in \mathcal{M}_{0,T}^B\times
L^2\left(\mathcal{F}^{B}_{\tau}; \Theta\right)$ be given and $(a_u,
p_u,X_u)\in\mathcal{J}^{1,2,+}_{h}u(\tau,\xi)$. We assume that $u$
is a stochastic subsolution of
$\mathcal{SPVI}^{(f,g,h,\chi,\varphi,\psi)}$ \eqref{SPVI}, which
means that
\begin{eqnarray*}
u(\tau,\xi)&\in & {\rm Dom}(\varphi),\;\;\;\;\;\; \forall\; (\tau,\xi)\in\mathcal{M}_{0,T}^B\times L^2\left(\mathcal{F}^{B}_{\tau};\Theta\right),\;\;\; \P\mbox{-a.s.},\\
u(\tau,\xi)&\in & {\rm Dom}(\psi),\;\;\;\;\;\; \forall\;
(\tau,\xi)\in\mathcal{M}_{0,T}^B\times
L^2\left(\mathcal{F}^{B}_{\tau};\partial\Theta\right),\;\;\;
\P\mbox{-a.s.},
\end{eqnarray*}
and at any $(\tau,\xi)\in\mathcal{M}_{0,T}^B\times L^2\left(\mathcal{F}^{B}_{\tau};\Theta\right)$, it holds $\P$-a.s.
\begin{itemize}
\item[(a)] on the event $\left\{0<\tau<T\right\}\cap\left\{\xi\in\Theta\right\}$
\begin{equation*}
V_{f}(\tau,\xi,a_u,p_u,X_u)+\varphi'_l(u(\tau,\xi)-\frac{1}{2}(h\partial_uh)(\tau,\xi,u(\tau,\xi))\leq 0;
\end{equation*}
\item[(b)] on the event $\left\{0<\tau<T\right\}\cap\left\{\xi\in
\partial \Theta\right\}$

$\displaystyle\min\Big(V_{f}(\tau,\xi,a_u,p_u,X_u)+\varphi'_l(u(\tau,\xi)-\frac{1}{2}(h\partial_u
h)(\tau,\xi,u(\tau,\xi)),$
\begin{align*}
\langle\nabla
\phi(\xi),p_u\rangle-g(\tau,\xi,u(\tau,\xi))+\psi'_l(u(\tau,\xi))\Big)\leq
0. \label{E:viscosity01}
\end{align*}
\end{itemize}
Then, according to Proposition 4.7, there exist $(a_v,
p_v,X_v)\in\mathcal{J}^{1,2,+}_{0}v(\tau,\xi)$, such that on the
event $\left\{0<\tau<T\right\}\cap\left\{\xi\in\Theta\right\},$
\begin{eqnarray*}
&&-D_y\eta(\tau,\xi,v(\tau,\xi))a_v- D_y\eta(\tau,\xi,v(\tau,\xi))\frac{1}{2}Tr(\sigma\sigma^*(\xi)X_v)\\
&&-\frac{1}{2}Tr(\sigma\sigma^*(\xi)D_{xx}\eta(\tau,\xi,v(\tau,\xi))-\frac{1}{2}D_{yy}\eta(\tau,\xi,v(\tau,\xi))|\sigma^*(\xi)p_v|^2\\
&&-\langle\sigma^{*}(\xi)D_{xy}\eta(\tau,\xi,v(\tau,\xi)),\sigma^*(\xi)p_v\rangle-\langle D_{x}\eta(\tau,\xi,v(\tau,\xi)),b(\xi)\rangle\\
&&-f(\tau, \xi, \eta(\tau,\xi,v(\tau,\xi)),\sigma^{*}(\xi)D_{x}\eta(\tau,\xi,v(\tau,\xi))+D_y\eta(\tau,\xi,v(\tau,\xi))\sigma^*(\xi)p_v)\\
&&-\langle D_{y}\eta(\tau,\xi,v(\tau,\xi))p_v,b(\xi)\rangle\\
&\leq&
-\varphi'_{l}(\eta(\tau,\xi,v(\tau,\xi)))+\frac{1}{2}(h\partial
h)(\tau,\xi,\eta(\tau,\xi,v(\tau,\xi))).
\end{eqnarray*}
Since $D_y\eta(t,x,y)>0,\; \forall\, (t,x,y)$ we define the random field $\widetilde{f}$ by
\begin{eqnarray*}
\widetilde{f}(t,x,y,z)&=&\frac{1}{D_y\eta(t,x,y)}\left[f(t,x,\eta(t,x,y),\sigma^{*}(x)D_x\eta(t,x,y)+D_y\eta(t,x,y)z)-\frac{1}{2}(h\partial_u h)(t,x,\eta(t,x,y))\right.\\
&&\left.+L_x\eta(t,x,y)+\langle\sigma^*(x)D_{xy}\eta(t,x,y),z\rangle+\frac{1}{2}D_{yy}\eta(t,x,y)|z|^2\right].
\end{eqnarray*}
We obtain
\begin{eqnarray*}
V_{\widetilde{f}}(\tau,\xi,a_v,p_v,X_v)+\frac{\varphi'_{l}(\eta(\tau,\xi,v(\tau,\xi)))}{D_y\eta(\tau,\xi,v(\tau,\xi))}\leq 0
\end{eqnarray*}
and on the event $\left\{0<\tau<T\right\}\cap\left\{\xi\in
\partial \Theta\right\}$

$\displaystyle \langle\nabla
\phi(\xi),p_u\rangle-g(\tau,\xi,u(\tau,\xi))+\psi'_l(u(\tau,\xi))$
\begin{eqnarray*}
&=&\langle\nabla \phi(\xi),D_{y}\eta(\tau,\xi,v(\tau,\xi))p_v\rangle
+\langle\nabla \phi(\xi),D_{x}\eta(\tau,\xi,v(\tau,\xi))\rangle\\
&&-g(\tau,\xi,\eta(\tau,\xi,v(\tau,\xi)))+\psi'_l(\eta(\tau,\xi,v(\tau,\xi)))\\
&=&D_{y}\eta(\tau,\xi,v(\tau,\xi))\langle\nabla \phi(\xi),p_v\rangle\\
&&-D_{y}\eta(\tau,\xi,v(\tau,\xi))\widetilde{g}(\tau,\xi,\eta(\tau,\xi,v(\tau,\xi))),
\end{eqnarray*}
where
\begin{eqnarray*}
\widetilde{g}(t,x,y)=\frac{1}{D_{y}\eta(t,x,y)}(g(t,x,y)-\langle\nabla\phi(\xi),D_{x}\eta(t,x,y)\rangle).
\end{eqnarray*}
Recall again that $D_y\eta(t,x,y)>0,\; \forall\, (t,x,y)$, we get
\begin{eqnarray*}
\min\left(V_{\widetilde{f}}(\tau,\xi,a_v,p_v,X_v)+\frac{\varphi'_l(\eta(\tau,\xi,v(\tau,\xi)))}
{D_y\eta(\tau,\xi,v(\tau,\xi))}
,\langle \nabla\phi(\xi),D_x\eta(\tau,\xi,v(\tau,\xi)p_v)\rangle\right.\\\\
\left.-\widetilde{g}(\tau,\xi,v(\tau,\xi))+\frac{\psi'_l(\eta(\tau,\xi,v(\tau,\xi)))}{D_y\eta(\tau,\xi,v(\tau,\xi))}\right)&\leq& 0.
\end{eqnarray*}

\end{proof}
\section{Probabilistic representation result for stochastic viscosity solution to SPVI}
The main objective of this section is to show how a semi-linear SPVI
with Neumann-Dirichlet condition associated to the coefficients
$(f,h,g,\chi,\varphi,\psi)$ is related to BDSGVI \eqref{equation1}
in the Markov framework.

We now introduce a class of reflected diffusion processes. Let us recall $\Theta$ be an
open connected bounded subset of $\R^d$, which is such that for a
function $\phi\in C^{2}_{b}(\R^d),\ \Theta=\{\phi>0\},\; \partial\Theta=\{\phi=0\}$. It follows from the results in Lions, Sznitman [6] (see also Saisho [10])
that for each $(t,x)\in[0,T]\times\overline{\Theta}$ there exists a unique pair of ${\bf F}^{W}$ progressively measurable continuous processes $\{X^{t,x}_s,A^{t,x}_s;  s\geq 0\}$, with values in
$\overline{\Theta}\times\R_+$, such that
:\allowdisplaybreaks
\begin{eqnarray}
X_s^{t,x} = x+\int^{s\vee t}_t b\left(X_r^{t,x}\right)\,{\rm
d}r+\int^{s\vee t}_t
\sigma\left(X_r^{t,x}\right)\,{\rm
d}{W}_r+\int^{s\vee t}_t \nabla \phi
\left(X_r^{t,x}\right)\,{\rm
d}A_r^{t,x}, \quad
\forall\, s\in [0,T].\label{rSDE}
\end{eqnarray}
Let notice that the above assumptions imply that there exists
a constant $\alpha >0$ such that for any $x\in\partial\Theta,\, x'\in\Theta$
\begin{eqnarray}
|x-x'|^2+\alpha\langle x'-x,\phi(x)\rangle\geq 0.\label{boundaryineq}
\end{eqnarray}
We have

\begin{proposition}\label{P:continuity00}
There exists a constant $ C>0 $ such that for all $0\leq t<t'\leq T$
and $x,\,x'\in \overline{\Theta}$,\  the following inequalities
hold: for any $p>4$
\begin{eqnarray}
\mathbb{E}\left[\sup_{0\leq s\leq
T}\left|X^{t,x}_{s}-X^{t',x'}_{s}\right|^p\right] \leq
C\left[ |t'-t|^{p/2}+|x-x'|^{p}\right]\label{continuity1}
\end{eqnarray}
and
\begin{eqnarray}
\mathbb{E}\left[\sup_{0\leq s\leq
T}\left|A_s^{t,x}-A_{s}^{t',x'}\right|^p\right]\leq C\left[
|t'-t|^{p/2}+|x-x'|^{p}\right].\label{continuity2}
\end{eqnarray}
Moreover, for all $p\geq 1$,
there exists a positive constant $C_p$ such that for all
$(t,x)\in[0,T]\times \overline{\Theta}$,
\begin{eqnarray}
\mathbb{E}\left(\left|A_s^{t,x}\right|^p \right) \leq C_p(1+t^p)\label{bound1}
\end{eqnarray} and
for each $\mu$, $t<s<T$, there exists a positive constant $C(\mu,t)$
such that for all $x\in\overline{\Theta}$,
\begin{eqnarray}
\mathbb{E}\left(\displaystyle e^{\mu A_{s}^{t,x}}\right) \leq C(\mu,t).\label{bound2}
\end{eqnarray}
\end{proposition}
\begin{proof}
This proof follows the similar argument used in \cite{PZ}.  We apply Itô's
formula to the semimartingale
\begin{eqnarray*}
\exp\left[-\frac{p}{\alpha}\left(\phi(X^{t,x}_s)+\phi(X^{t',x'}_s)\right)\right]\left|X^{t,x}_s-X^{t',x'}_s\right|^{p}.
\end{eqnarray*}
Hence exploit the inequality \eqref{boundaryineq} and standard SDE estimates we obtain
\begin{eqnarray*}
\E\left(\sup_{0\leq s\leq T}|X^{t,x}_s-X^{t',x'}_s|^p\right)\leq C\left(|t-t'|^{p/2}+|x-x'|^p+\E\int^{T}_{0}|X^{t,x}_s-X^{t',x'}_s|^pds\right).
\end{eqnarray*}
Then inequality \eqref{continuity1} follows from Gronwall's lemma. Next, by Itô formula we have
\begin{eqnarray*}
A^{t,x}_s=\phi(X^{t,x}_s-\phi(x)-\int_t^{t\vee s}L\phi(X^{t,x}_rdr-\int_t^{t\vee s}\nabla\phi(X^{t,x}_r\sigma(X^{t,x}_r)dW_r,
\end{eqnarray*}
where $L$ is defined by \eqref{operateur}. From this identity and  inequality \eqref{continuity1}, we deduce easily the the inequalities \eqref{continuity2} and \eqref{bound1}.
\end{proof}
Under assumptions $(H1)$--$(H7)$, it
follows from Theorem 3.2 that, for all $(t, x)\in[0, T
]\times\overline{\Theta}$, there exists a unique triplet $(Y^{t,x},
Z^{t,x}, U^{t,x},V^{t,x})$ such that
\begin{eqnarray}
Y^{t,x}_s+\int_s^TU^{t,x}_r\,{\rm
d}r+\int_s^TV^{t,x}_r\,{\rm
d}A^{t,x}_r& =&\chi(X^{t,x}_T)+\int_s^Tf(r,X^{t,x}_r,Y^{t,x}_r,Z^{t,x}_r)\,{\rm
d}r+\int_s^Tg(r,X^{t,x}_r,Y^{t,x}_r)\,{\rm
d}A^{t,x}_r\nonumber\\
&&+\int_s^Th(r,X^{t,x}_r,Y^{t,x}_r)\,{\rm d}B_r-\int_s^TZ^{t,x}_r\,{\rm d}W_r,\ t\leq s \leq T;\label{eqmarkov1}
\end{eqnarray}
and
\begin{eqnarray}
(Y^{t,x}_s,U^{t,x}_s)\in \partial \varphi, \ \,{\rm d}\overline{\mathbb{P}}\otimes
\,{\rm d}s, \ (Y^{t,x}_s,V^{t,x}_s)\in \partial \psi, \ \,{\rm d}\overline{\mathbb{P}}\otimes
dA(\bar{\omega})_s
  \mbox{-a.e. on}\ [t,T].\label{eqmarkov2}
\end{eqnarray}
We extend processes $Y^{t,x},\, Z^{t,x},\,U^{t,x},\, V^{t,x}$ on $[0,T]$ by putting
\begin{eqnarray*}
Y^{t,x}_s=Y^{t,x}_t,\, Z^{t,x}_s=0,\,\,U^{t,x}_s=0,\,\, V^{t,x}_s=0,\;\; s\in [0,t].
\end{eqnarray*}
The following regularity result generalizes the Kolmogorov continuity criterion to BDSGVI:
\begin{proposition}\label{Prop}
Let the ordered triplet $(Y^{t,x}_s, U^{t,x}_s,V^{t,x}_s,
Z^{t,x}_s)$ be the unique solution of the BDSGVI \eqref{eqmarkov1}.
Then, the random field $(s, t, x)\mapsto Y^{t,x}_s,\; (s, t,
x)\in[0, T ]\times[0, T ]\times\overline{\Theta}$, is a.s.
continuous.
\end{proposition}
\begin{proof}
Let $(t,x)$ and $(t',x')$ be two elements of $[0,T]\times \overline{\Theta}$. It follows from Itô formula applied to $|Y^{t,x}_s-Y^{t',x'}_s|^p$ with $p>4$ combined with the arguments used in \cite{BM} (see Proposition 3.5) and \cite{Bal} (see Proposition 4.3) that, for $0\leq s\leq T$,
\begin{eqnarray*}
\E\left(\sup_{0\leq s\leq T}\left\vert
Y^{t,x}_{s}-Y^{t',x'}_{s}\right\vert^p\right)&\leq& C\left[\E\left(\sup_{0\leq s\leq T}\left|X^{t,x}_s-X^{t',x'}_s\right|^p\right)+\left(\E\sup_{0\leq s\leq T}|A^{t,x}_s-A_s^{t',x'}|^p\right)^{1/2}\right].
\end{eqnarray*}
Next, using Proposition 5.1 one can derive
\begin{eqnarray*}
\E\left(\sup_{0\leq s\leq T}\left\vert
Y^{t,x}_{s}-Y^{t',x'}_{s}\right\vert^p\right)\leq C(|t-t'|^{p/2}+|x-x'|^{p}+|t-t'|^{p/4}+|x-x'|^{p/2}).
\end{eqnarray*}
Therefore, il suffice to choose $p=\gamma$ convenably to get
\begin{eqnarray*}
\E\left(\sup_{0\leq s\leq T}\left\vert
Y^{t,x}_{s}-Y^{t',x'}_{s}\right\vert^{\gamma}\right)\leq C(|t-t'|^{1+\beta}+|x-x'|^{d+\delta}).
\end{eqnarray*}
We conclude from the last estimate, using Kolmogorov's lemma, that $\displaystyle{\{Y_s^{t,x}, s,t\in[0,T], x\in\overline{\Theta}\}}$
has an a.s. continuous version.
\end{proof}
Let us define
\begin{eqnarray}
u(t,x)=Y^{t,x}_t
\end{eqnarray}
which is random field such that $u(t,x)$ is $\mathcal{F}^{B}_{t,T}$-measurable for each $(t,x)\in[0,T]\times\overline{\Theta}$.

We are now ready to derive the main result of this section.
\begin{theorem}
Let the assumptions $({H1})$--$({H7})$ be satisfied. Then,
the function $u(t, x)$ defined above is a stochastic viscosity
solution of $\mathcal{SPVI}^{(f,g,h,\chi,\varphi,\psi)}$
\eqref{SPVI}.
\end{theorem}
\begin{proof}
First, since $u(t,x)=Y^{t,x}_t$, it follows from Proposition 5.2 that $u\in C(\mathcal{F}^{B},[0,T]\times \overline{\Theta})$. Thus it remains to show that $u$ is the stochastic viscosity solution to $\mathcal{SPVI}^{(f,g,h,\chi,\varphi,\psi)}$. In other word, using
Corollary 4.8, it suffices to prove that $v(t, x) =\varepsilon(t, x, u(t, x))$ satisfies \eqref{E:def2}--\eqref{E:viscosity02} and
\eqref{E:def02}--\eqref{E:viscosity002}. In this fact, we are going to use the Yosida approximation of \eqref{eqmarkov1},
which was studied in Section 3.  For each $(t,x)\in[0,T]\times\overline{\Theta},\; \delta>0$, let $\{(Y^{t,x,\delta}_{s},Z^{t,x,\delta}_{s}),\,\ 0\leq s\leq T\}$ denote the solution of the following GBDSDE:
\begin{eqnarray}
&&Y^{t,x,\delta}_s+\int_s^T\nabla\varphi_{\delta}(Y^{t,x,\delta}_r)\,{\rm
d}r
+\int_s^T\nabla\psi_{\delta}(Y^{t,x,\delta}_r)\,{\rm
d}A_r\nonumber\\
&=&\chi(X^{t,x}_T)+\int_s^Tf(r,X^{t,x}_r,Y^{t,x,\delta}_r,Z^{t,x,\delta}_r)\,{\rm
d}r+\int_s^Tg(r,X^{t,x}_r,Y^{t,x,\delta}_r)\,{\rm
d}A_r\nonumber\\
&&+\int_s^T h(r,X^{t,x}_r,Y^{t,x,\delta}_r) \,{\rm
d}B_r-\int_s^TZ^{t,x,\delta}_r\,{\rm d}W_r,\ t\leq s \leq
T.\label{eqmarkovapp}
\end{eqnarray}

Define $Y^{t,x,\delta}_t=u^{\delta}(t,x)$,  it is well known (see
Theorem 4.7, \cite{Bal}) that the function $v^{\delta}(t, x) = \varepsilon(t,
x, u^{\delta}(t, x))$ is an $\omega$-wise viscosity solution to the
following SPDE with nonlinear Dirichlet-Neumann boundary condition
\begin{eqnarray}
\left\{
\begin{array}{l}
{\rm(i)}\;\displaystyle\left(\frac{\partial v^{\delta}}{\partial
t}(t,x)-\left[
Lv^{\delta}(t,x)+\widetilde{f}_{\delta}(t,x,v^{\delta}(t,x),\sigma^{*}(x)\nabla
v^{\delta}(t,x))\right]\right)=0,\,\,\
(t,x)\in[0,T]\times\Theta,\\\\
{\rm(ii)}\;\displaystyle\frac{\partial v^{\delta}}{\partial
n}(t,x)+\widetilde{g}_{\delta}(t,x,v^{\delta}(t,x))=0,\,\,\ (t,x)\in[0,T]\times\partial\Theta, \\\\
{\rm(iii)}\; v(T,x)=\chi(x),\,\,\,\,\,\,\ x\in\overline{\Theta},
\end{array}\right.\label{SPDE}
\end{eqnarray}
where
$$\widetilde{f}_{\delta}(t,x,y,z)=\widetilde{f}(t,x,y,z)-\frac{\nabla\varphi_{\delta}(\eta(t,x,y))}{D_y\eta(t,x,y)}\;\;\;\mbox{and}
\;\;\;\widetilde{g}_{\delta}(t,x,y)=\widetilde{g}(t,x,y)-\frac{\nabla\psi_{\delta}(\eta(t,x,y))}{D_y\eta(t,x,y)}.$$
However from the results of the previous section, one can proved with no more difficulties (it suffice to show that
$\displaystyle{\E\left[\sup_{t\leq s\leq T}\sup_{x\in\overline{\Theta}}|,Y^{t,x,\varepsilon}_s-Y^{t,x,\delta}_s|^2\right]\leq C(\varepsilon+\delta)}$) that, for each $(t, x)\in[0, T ]\time\Theta$, along a subsequence, $v^\delta(t,x)$ converge to $v(t,x)$ almost surely as $\delta $ goes to 0. Moreover, since $v^\delta$ and $v$ are continuous, it follows from Dini's theorem that the above convergence is uniform on $t$ on compact.

Let $\omega\in\Omega$ be fixed such
\begin{eqnarray*}
|v^{\delta}(\tau(\omega),\xi(\omega))-v(\tau(\omega),\xi(\omega))|\rightarrow 0\;\;\; \mbox{as}\;\;\; \delta\rightarrow 0,
\end{eqnarray*}
and consider $(a_v,p_v,X_v)\in\mathcal{J}^{1,2,+}_{0}(v(\tau(\omega),\xi(\omega)))$. Then, it follows from Crandall-
Ishii-Lions \cite{CIL} that there exist sequences
\begin{eqnarray*}
\left\{
\begin{array}{ll}
\delta_n(\omega)\searrow 0,\\\\
(\tau_n(\omega),\xi_n(\omega))\in[0,T]\times\overline{\Theta},\\\\
(a_v^{n},p_v^{n},X_v^{n})\in\mathcal{J}^{1,2,+}_{0}(v^{\delta_n}(\tau_n(\omega),\xi_n(\omega)))
\end{array}
\right.
\end{eqnarray*}
such that
\begin{eqnarray*}
(\tau_n(\omega),\xi_n(\omega),a_v^{n},p_v^{n},X_v^{n},v^{\delta_n}(\tau_n(\omega),
\xi_n(\omega)))\rightarrow
(\tau(\omega),\xi(\omega),a_v,p_v,X_v,v(\tau(\omega),\xi(\omega))), \
\mbox{as} \ n\to \infty.
\end{eqnarray*}
Since $v^{\delta_n}(\omega,\cdot,\cdot)$ is a (deterministic) viscosity
solution to the PDE
$(\widetilde{f}_{\delta_n}(\omega,\cdot,\cdot,\cdot),0,\widetilde{g}_{\delta_n}(\omega,\cdot,\cdot),\chi)$,
we obtain
\begin{itemize}
\item [(a)] $(\tau_n(\omega),\xi_n(\omega))\in [0,T]\times\Theta$
\begin{eqnarray}
V_{\widetilde{f}_{\delta_n}(\omega)}(\tau_n(\omega),\xi_n(\omega),a^n_v,X^n_v,p^n_v)+\frac{\nabla\varphi_{\delta_n}(\eta(\tau_n(\omega),\xi_n(\omega),v^{\delta_n}
(\tau_n(\omega),\xi_n(\omega))))}
{D_y\eta(\tau_n(\omega),\xi_n(\omega),v^{\delta_n}(\tau_n(\omega),\xi_n(\omega)))}\leq 0,\label{V01}
\end{eqnarray}
\item [(b)]$(\tau_n(\omega),\xi_n(\omega))\in [0,T]\times\partial\Theta$

$\displaystyle\min\Bigg(V_{\widetilde{f}_{\delta_n}(\omega)}(\tau_n(\omega),\xi_n(\omega),a^n_v,X^n_v,p^n_v)+\frac{\nabla\varphi_{\delta_n}(\eta(\tau_n(\omega),
\xi_n(\omega),v^{\delta_n}(\tau_n(\omega),\xi_n(\omega))))}{D_y\eta(\tau_n(\omega),\xi_n(\omega),v^{\delta_n}
(\tau_n(\omega),\xi_n(\omega)))},$
\begin{eqnarray}
&&\langle\nabla\phi(\xi_n),D_x\eta(\tau_n(\omega),\xi_n(\omega),v^{\delta_n}(\tau_n(\omega),\xi_n(\omega))p^n_v)\rangle
-\widetilde{g}_{\delta_n}(\omega)(\tau_n(\omega),\xi_n(\omega),v^{\delta_n}(\tau_n(\omega),\xi_n(\omega)))\nonumber\\\nonumber\\
&&~~~~~~~~~~~~~~~~~~~~~~~~~~~~~~~+\frac{\nabla\psi_{\delta^n}(\eta(\tau_n(\omega),\xi_n(\omega)
,v^{\delta_n}(\tau_n(\omega),\xi_n(\omega))))}
{D_y\eta(\tau_n(\omega),\xi_n(\omega),v^{\delta_n}(\tau_n(\omega),\xi_n(\omega)))}\Bigg)\leq
0\label{V02}.
\end{eqnarray}
\end{itemize}
To simplify the notation, we remove the dependence of $\omega$ in
the sequel. Let $y\in {\rm Dom}(\varphi)\cap {\rm Dom}(\psi)$ such
that $y\leq u(\tau,\xi)=\eta(\tau,\xi,v(\tau,\xi))$. The ucp
convergence $v^{\delta_n}\rightarrow v$ implies that there exists
$n_0 > 0$ such that  $y<\eta(\tau,\xi,v(\tau,\xi)),\;\; \forall\,
n\geq n_0$. Therefore, from \eqref{V01} and \eqref{V02}, it follows
that

$\displaystyle
(\eta(\tau,\xi,v(\tau,\xi))-y)V_{\widetilde{f}}(\tau_n,\xi_n,a_v,X_v,p_v)$
$$\leq
\left[-\varphi_{\delta}(J_{\delta_n}(\eta(\tau,\xi,v^{\delta_n}(\tau,\xi))))+\varphi(y)\right]\frac{1}
{D_y\eta(\tau_n,\xi_n,v^{\delta_n}(\tau_n,\xi_n))},\label{V1}$$ and

$\displaystyle
\min\Bigg((\eta(\tau,\xi,v(\tau,\xi))-y)V_{\widetilde{f}}(\tau_n,\xi_n,a_v,X_v,p_v)
+\frac{\varphi_{\delta}(J_{\delta}(\eta(\tau,
\xi,v^{\delta_n}(\tau,\xi))))-\varphi(y)}{D_y\eta(\tau,\xi,v^{\delta_n}
(\tau,\xi))},$
\begin{eqnarray*}
&&\left(\eta(\tau,\xi,v(\tau,\xi))-y\right)\left
[\langle\nabla\phi(\xi),D_x\eta(\tau_n(\omega),\xi_n,v^{\delta_n}(\tau,\xi_n)p^n_v)\rangle
-\widetilde{g}(\tau_n,\xi_n,v^{\delta_n}(\tau_n,\xi_n))\right]\nonumber\\\nonumber\\
&&~~~~~~~~~~~~~~~~~~~~~~~~~~~~~~~+\frac{\psi_{\delta}(\bar{J}_{\delta}(\eta(\tau,
\xi,v^{\delta_n}(\tau,\xi))))-\psi(y)}{D_y\eta(\tau,\xi,v^{\delta_n}
(\tau,\xi))}\Bigg)\leq 0\label{V2}.
\end{eqnarray*}
Taking the limit in this last inequality, for all
$y\leq\eta(\tau,\xi,v(\tau,\xi))$, we get
$$(V_{\widetilde{f}}(\tau,\xi,a_v,X_v,p_v)\leq -\frac{\varphi(\eta(\tau,\xi,v(\tau,\xi)))-\varphi(y)}{\eta(\tau,\xi,v(\tau,\xi))-y)}\frac{1}
{D_y\eta(\tau,\xi,v(\tau,\xi))},\label{V1}
$$
and

$\displaystyle\min\Bigg(V_{\widetilde{f}}(\tau,\xi,a_v,X_v,p_v)+\frac{\varphi(\eta(\tau,
\xi,v(\tau,\xi)))-\varphi(y)}{(\eta(\tau,\xi,v(\tau,\xi))-y)D_y\eta(\tau,\xi,v(\tau,\xi))},$
\begin{eqnarray*}
&&\left[\langle\nabla\phi(\xi),D_x\eta(\tau,\xi,v(\tau,\xi)p_v)\rangle
-\widetilde{g}(\tau,\xi,v(\tau,\xi))\right]+\frac{\psi(\eta(\tau(\omega),
\xi,v(\tau,\xi)))-\psi(y)}{(\eta(\tau,\xi,v(\tau,\xi))-y)D_y\eta(\tau(\omega),\xi,v
(\tau,\xi))}\Bigg)\leq 0\label{V2},
\end{eqnarray*}
which implies that

\begin{eqnarray*}
&&(V_{\widetilde{f}}(\tau,\xi,a_v,X_v,p_v)+\frac{\varphi'_l(\eta(\tau,\xi,v(\tau,\xi)))}
{D_y\eta(\tau,\xi,v(\tau,\xi))}\leq 0,\label{V1}
\end{eqnarray*}
and

$\displaystyle\min\Bigg(V_{\widetilde{f}}(\tau,\xi,a_v,X_v,p_v)+\frac{\varphi'_l(\eta(\tau,
\xi,v(\tau,\xi)))}{D_y\eta(\tau,\xi,v(\tau,\xi))},$
\begin{eqnarray*}
 &&\langle\nabla\phi(\xi),D_x\eta(\tau,\xi,v(\tau,\xi)p_v)\rangle
-\widetilde{g}(\tau,\xi,v(\tau,\xi))+\frac{\psi'_l(\eta(\tau,
\xi,v(\tau,\xi)))}{D_y\eta(\tau,\xi,v (\tau,\xi))}\Bigg)\leq
0\label{V2},
\end{eqnarray*}
and yields that $v$ satisfies \eqref{E:def2} and
\eqref{E:viscosity02}. Then, it follows from Corollary
\ref{corollary4.8} that $u$ is a stochastic viscosity subsolution of
$\mathcal{SPVI}^{(f,g,h,\chi,\varphi,\psi)}$ \eqref{SPVI}. By
similar arguments, one can prove that $u$ is a stochastic viscosity
supersolution of $\mathcal{SPVI}^{(f,g,h,\chi,\varphi,\psi)}$
\eqref{SPVI} and completes the proof.
\end{proof}
{\textbf {Acknowledgements}} The authors wish to thank the two anonymous
referees and the associate editor for their valuable comments, correcting errors and improving written language.

\label{lastpage-01}
\end{document}